\documentclass[11pt]{article}
\usepackage{extpfeil}
\usepackage[all]{xy}
\usepackage{amssymb} 
\usepackage[margin=1in]{geometry}
\usepackage{euscript}
\usepackage{mathtools}
\usepackage{amsthm}
\usepackage{verbatim}
\usepackage{mathrsfs}
\usepackage{url}
\usepackage{color}

\numberwithin{equation}{section}

\theoremstyle{plain}
	\newtheorem{thm}[equation]{Theorem}
	\newtheorem{prop}[equation]{Proposition}

	\newtheorem{cor}[equation]{Corollary}
	
	\newtheorem{lem/defn}[equation]{Lemma/Definition}

\theoremstyle{definition}
	\newtheorem{defn}[equation]{Definition}
	\newtheorem{ex}[equation]{Example}
	
	\newtheorem{const}[equation]{Construction}

\theoremstyle{remark}
	\newtheorem{rem}[equation]{Remark}

\def\nc{\newcommand}
\def\on{\operatorname}

\nc{\edit}[1]{\marginpar{\footnotesize{#1}}}

\nc{\C}{\mathbb{C}}
\nc{\Z}{\mathbb{Z}}
\nc{\PP}{\mathbb{P}}
\nc{\R}{\mathbb{R}}
\nc{\F}{\mathbb{F}}
\nc{\AAA}{\mathbb{A}}
\nc{\LL}{\mathbb{L}}
\nc{\OO}{\mathcal{O}}

\nc{\X}{\EuScript{X}}
\nc{\sZ}{\EuScript{Z}}

\nc{\id}{{\on{id}}}
\nc\Hom{{\on{Hom}}}
\nc\cone{{\on{cone}}}
\nc{\Rep}{{\on{Rep}}}
\nc\Ob{{\on{Ob}}}
\nc\Spec{{\on{Spec}}}
\nc\Mod{{\on{Mod}}}
\nc\coMod{{\on{coMod}}}
\nc\Perf{{\on{Perf}}}
\nc\End{{\on{End}}}
\nc{\into}{\hookrightarrow}
\nc{\tr}{\on{tr}}
\nc{\ev}{\on{ev}}
\nc{\im}{\on{im}}
\nc{\cogroupoidzerosphere}{S^{0, \bullet}_{fil}}
\nc{\Mot}{\on{Mot}}
\nc{\filstack}{\AAA^1 / \GG_m}
\nc{\pt}{\on{pt}}
\nc{\formalgroup}{\widehat{\GG}}
\nc{\coker}{\on{coker}}
\nc{\rk}{\on{rank}}
\nc{\TOP}{\on{Top}_{\mathbb{C}}^{s}}
\nc{\gr}{\on{gr}}
\nc{\Catperf}{\text{Cat}^{\text{perf}}}
\nc{\Sym}{\on{Sym}}
\nc{\xra}{\xrightarrow}
\nc{\lra}{\xleftarrow}
\nc{\Bet}{\mathbf{Betti}_{X}}
\nc{\spectralift}{\mathsf{Fix}^{un}_{\formalgroup}}
\nc{\Fix}{\mathsf{Fix}}
\nc{\codim}{\on{codim}}
\nc{\Fred}{\on{Fred}}
\nc{\colim}{\on{colim}}
\nc{\KK}{{\bf K}}
\nc{\Sp}{\on{Sp}}
\nc{\onto}{\twoheadrightarrow}
\nc{\A}{\mathbb{A}}
\nc{\Aff}{\on{Aff}}
\nc{\SH}{\on{SH}}
\nc{\QCoh}{\on{QCoh}}
\nc{\filteredloopspaceX}{\mathcal{L}_{fil}(X)}
\nc{\loopspaceX}{\mathcal{L}(X)}
\nc{\Alg}{\on{Alg}}
\nc{\Br}{\on{Br}}
\nc{\loopsX}{\mathcal{L}X}
\nc{\shiftedtangent}{\mathcal{T}_X[-1]}
\nc{\ta}{\widetilde{\a}}
\nc{\Shv}{\on{Shv}}
\nc{\GG}{\mathbb{G}}
\nc{\red}{\color{red}}
\nc{\an}{\on{an}}
\nc{\D}{\on{D}}
\nc{\Pre}{\on{Pre}}
\nc{\circlegroupoid}{{S^{1, \bullet, \bullet}_{fil}} }
\nc{\qc}{\on{qc}}
\nc{\op}{\on{op}}
\nc{\shEnd}{{\mathcal End}}
\nc{\Sph}{\mathbb{S}}
\nc{\Top}{\on{Top}}
\nc{\Map}{\on{Map}}
\nc{\Vect}{\on{Vect}}
\nc{\holim}{\on{holim}}

\def\A{\mathcal{A}}

\def\a{\alpha}

\def\Perf{\on{Perf}}

\def\Sp{\on{Sp}}

\def\Y{\EuScript{Y}}

\nc{\W}{\mathbb{W}}

\def\QCoh{\on{QCoh}}

\title{Cogroupoid structures on the circle and the Hodge degeneration}
\author{Tasos Moulinos}
\date{}

\begin{document}
\vspace{18mm} \setcounter{page}{1} \thispagestyle{empty}

\maketitle

\begin{abstract}
We exhibit the Hodge degeneration from nonabelian Hodge theory as a $2$-fold delooping of the filtered loop space $E_2$-groupoid in formal moduli problems. This is an iterated groupoid object which in degree $1$ recovers the filtered circle $S^1_{fil}$ of \cite{moulinos2019universal}. 
 This exploits a hitherto unstudied additional piece of structure on the topological circle, that of an $E_2$-cogroupoid object in the $\infty$-category of spaces.  We relate this cogroupoid structure with the more commonly studied ``pinch map" on $S^1$, as well as the Todd class of the Lie algebroid $\mathbb{T}_{X}$; this is an invariant of a smooth and proper scheme $X$ that arises, for example, in the Grothendieck-Riemann-Roch theorem. In particular we relate the existence of non-trivial Todd classes for schemes to the failure of the pinch map to be formal in the sense of rational homotopy theory. Finally we record some consequences of this bit of structure at the level of Hochschild cohomology. 
\end{abstract}
\section{Introduction}
The de Rham cohomology of a ring or scheme comes equipped with a complete, decreasing filtration, known as the \emph{Hodge filtration}.
This has been studied in many capacities, and in particular occupies a paradigmatic role in the theory of Hodge structures. 
Carlos Simpson, in his work on nonabelian Hodge theory understood this filtration in the geometric language of stacks, using the paradigm of filtrations and geometric objects over the stack $\filstack$.  Using a deformation to the normal cone construction of Fulton-Macpherson (cf. \cite{fulton}), Simpson displayed the stack parametrizing \emph{$\lambda$-connections} as a 1-parameter degeneration of the de Rham stack $X_{dR}$, which itself parametrizes bundles with flat connection over a fixed scheme $X$. These $\lambda$-connections are exactly the objects which interpolate between Higgs bundles and bundles with flat connection, and give rise, upon passing to subcategories of  \emph{harmonic bundles}, to equivalences between the two structures. 

Meanwhile, there exists another one-parameter degeneration relating a \emph{sheared} version of the de Rham cohomology of a  scheme with its \emph{Hochschild homology}. This was studied in depth in \cite{moulinos2019universal} (see also \cite{raksit2020hochschild} for another perspective on the matter), wherein the authors constructed a filtration on Hochschild homology whose associated graded recovers the \emph{derived de Rham algebra}. This was accomplished algebro-geometrically, by way of the \emph{filtered loop space} $\filteredloopspaceX$; this is a relative derived scheme over $\filstack$ which base-changes to the loop space $\mathcal{L}X$ and to the shifted tangent bundle $\mathcal{T}_X[-1]$ thus recovering the $S^1$-equivariant HKR filtration on Hochschild homology. 

The purpose of this work is to relate these two constructions in the setting of derived geometry. Our main theorem can be stated as follows:

\begin{thm} \label{maintheorem}
Let $X$ be a derived scheme. Then  the filtered loop space $\filteredloopspaceX$ fits as the degree $(1,1)$ piece of a 2-groupoid $\mathcal{L}^{\bullet, \bullet}_{fil}(X)$ in formal derived stacks over $\filstack$. Taking the 2-fold delooping of this groupoid gives the following equivalence: 
$$
B^{(2)} \mathcal{L}^{\bullet, \bullet}_{fil}(X) \simeq X_{Hod}
$$
\end{thm}

As a corollary one obtains the following relationship between the de Rham space $X_{dR}$ and the loop space $\mathcal{L}X$:

\begin{cor}
The derived loop space $\mathcal{L}X$ admits a 2-groupoid structure whose iterated delooping is the de Rham space, i.e. there is an equivalence 
$$
B^{(2)}(\mathcal{L}X) \simeq X_{dR} 
$$
of formal derived stacks. 
\end{cor}

These statements follow directly from a canonical $E_2$-\emph{cogroupoid} (which we often call a $2$-cogroupoid) structure on the filtered loop space, $S^1_{fil}$. More precisely, the filtered loop space sits as the ``$(1,1)$ space of morphisms" of a bi-cosimplicial stack satisfying an analog of the Segal conditions (cf. \cite[Section 6.1]{lurie2009higher}). Given a stack $X$, the formation of mapping objects out of a $E_2$-cogroupoid object into $X$ give rise to a $E_2$-groupoid; in this particular case where the cogroupoid is $S^1_{fil}$, we obtain $\mathcal{L}^{\bullet, \bullet}_{fil}(X)$. 

We note the following interesting consequence about the filtered circle. In \cite{moulinos2019universal} it was verified that, when working over $\mathbb{Q}$, there is an equivalence 
$$
(S^1_{fil})^{u} := \eta^*(S^1_{fil}) \simeq B \GG_a \simeq  \iota (S^1_{fil}) =: S^1_{gr}.
$$
Here, $\eta: \Spec k \to \filstack$ denotes the inclusion of the ``generic" point of $\filstack$, while $\iota: B \GG_m \to \filstack$ denotes the inclusion of the ``closed" point; in the language of \cite{geometryoffiltr}, restriction of a stack $\EuScript{X} \to \filstack$ along these maps recovers the underlying and associated graded stacks of $\EuScript{X}$ respectively.  

Thus  $S^1_{fil}$ is a constant degeneration (see Section \ref{section formal groups todd} for this terminology) of stacks. This is not the case when one takes the $E_2$-cogroupoid structure into account:

\begin{cor}
The $E_2$-cogroupoid $S^{1, \bullet, \bullet}_{fil} \to \filstack$ is not a constant degeneration of $E_2$-cogroupoids. In particular the pullbacks of $S^{1, \bullet, \bullet}_{fil}$ along $\Spec k \to \filstack$ and $B \GG_m \to \filstack$ are not equivalent. 
\end{cor}

In order to prove theorem \ref{maintheorem} we work in the setting of formal geometry, and formal moduli problems. As a clue for why we find ourselves in this setting essentially at the outset, we remark that the derived  loop space $\mathcal{L}X$ is formally complete along $X$ when $X$ is a scheme. More generally, we will see that the filtered loop space obtains the structure of an $E_2$-groupoid  $ \mathcal{L}^{\bullet, \bullet}_{fil}(X)$ in the $\infty$-category of \emph{formal moduli problems over $X$}.  Meanwhile, the map $X \to X_{dR}$ to the de Rham space of $X$ is also a nil-isomorphism. In fact $X_{dR}$ is the final object in the category of formal moduli problems under $X$.

A key property which we exploit in the setting  of formal moduli problems is the following:
given a map $f: X \to Y$, which is a nil-isomorphism,   $Y$ may be recovered as the classifying space of the Čech nerve. In this sense, the class of  nil-isomorphisms play the same role within formal geometry as the class  of  effective epimorphisms in derived geometry. Due to all this, the relation between the Hodge degeneration and the filtered loop space is most easily distilled in this setting of formal moduli problems, which we shall briefly review in Section \ref{section formalmodulistuff}.

\subsection{The pinch map on $S^1$ and the Todd class}
In a very influential paper \cite{markarian2009atiyah}, N. Markarian described the Todd genus of a smooth proper scheme $X$ as an invariant volume form with respect to the Hopf algebra structure on Hochschild homology. This was made more precise in \cite{kondyrev2019equivariant}, in terms of the formal group structure on the derived loop space $\loopspaceX$. In this paper, we expand on this story by describing the group structure on $\loopspaceX$ as arising from an $E_1$-cogroupoid structure on the circle $S^1$. This $E_1$-cogroupoid structure, equivalently a \emph{cogroup} structure on $S^1$ as a \emph{pointed space}, is none other than the well known ``pinch map" which gives rise to the group structure on the fundamental groups of topological spaces. This $E_1$ cogroupoid, which exists since $S^1$ is a suspension of a space, can also be extracted as either the row or column $1$ sub-cogroupoid of the $E_2$-cogroupoid $S^{1,\bullet, \bullet}$
It is well known that over the rationals, the cochain dga  $C^*(S^1, k)$ is formal, so that there is an equivalence 
$$
C^*(S^1, k) \simeq H^*(S^1, k) 
$$
We show that this formality statement is not true at the level of $E_1$ cogroupoids, and in particular, this failure is measured by the Todd class:

\begin{thm}
The existence of non-trivial Todd classes implies that the cogroupoid  (equivalently pointed cogroup) structure on $S^1$ corresponding to the pinch map is not formal.  
\end{thm}

One can summarize the above by saying that the ``pinch map" on $S^1$ is not formal, even though the dga of  cochains on $S^1$ is well-known to be formal when working rationally. Furthermore, the difference between the induced groupoid structure on $C^*(S^1, k)$ and that on $H^*(S^1, k)$ manifests itself algebro-geometrically, by way of the two different group structures on the shifted tangent bundle $\mathcal{T}_{X}[-1]$, which in turn is measured by the Todd class of the tangent Lie algebroid $\mathbb{T}_{X}$.

\subsection{Remarks on Hochschild cohomology}
The cogroupoid structures on $S^1$ also manifest themselves at the level of Hochschild \emph{cohomology}. As we will remark in Section \ref{section Hochschild cohomology}, the well known $E_2$-algebra structure on Hochschild cohomology together with its (dual) HKR filtration arises from the $E_2$-cogroupoid $S^1_{fil}$. 

\begin{prop}
Let $X$ be a derived scheme. The $E_2$-cogroupoid $S^1_{fil}$ gives rise to filtration on $\on{HH}^{*}(X)$ compatible with its $E_2$-algebra structure.  
\end{prop}

Perhaps more surprisingly, we highlight that one only needs an $E_1$-cogroupoid structure on the zero sphere $S^0$ to recover the $E_2$-algebra structure on Hochschild cohomology. As the zero sphere is not a suspension of a space, there exists no cogroup structure on it; in general, there is no group structure on $\pi_0(X)$ of a topological space in general. However, by taking the conerve of the map $\emptyset \to *$ from the initial object to the final object in spaces, we recover a cogroupoid with $S^0$ in degree 1. 

We summarize the corresponding discussion in Section \ref{section Hochschild cohomology} with the following proposition
\begin{prop}
The cogroupoid $S_{fil}^{0, \bullet}$ (cf. Section \ref{section cogroupoid objects}) gives rise to a monoidal structure on 
\begin{equation} \label{svenchen}
   \QCoh(\Map_{\on{dStk_{\filstack}}}(S^0_{fil},X|_{\filstack})), 
\end{equation}
the $\infty$-category of quasi coherent sheaves on the mapping stack $\Map_{\on{dStk_{\filstack}}}(S^0_{fil},X|_{\filstack})$. This specializes, by pulling back along the generic fiber, to the convolution  monoidal structure on 
$$
\QCoh(X \times X) \simeq \on{Fun}_{k}(\QCoh(X),\QCoh(X) )
$$
Forming endomorphisms of the unit in (\ref{svenchen}) gives an $E_2$-algebra $\on{HH}_{fil}^{*}(X)$ in filtered complexes $\on{Fil}(\Mod_k)$, which specializes along the generic fiber of $\filstack$ to Hochschild cohomology. 
\end{prop}

Finally, we remark that one may iterate the constructions, to obtain $E_{n+1}$ cogroupoids $S^{n, \bullet, \cdots, \bullet}_{fil}$ over $\filstack$. These give rise, verbatim to the filtrations on iterated Hochschild cohomology $HH^*_{E_n}$ compatible with the $E_{n+1}$-algebra structure.  

\vskip \baselineskip
\noindent {\bf Conventions.}
In general we work over a characteristic zero base $k$, although many of the constructions work more generally. In another vein, we work freely in the setting of $\infty$-categories and higher algebra from \cite{lurie2017higher}. Similarly, we heavily utilize at times the formalism of formal geometry in \cite{gaitsgory2019study, gaitsgory2020study}, which in particular does depend on the fact we work over $\mathbb{Q}$. 
\vskip\baselineskip
\noindent{\bf Acknowledgements.} I would like to thank Bertrand To\"{e}n for various conversations and ideas which led to the content of paper. I would also like to thank Joost Nuiten for helpful conversations. This work is supported by the grant NEDAG ERC-2016-ADG-741501. 

\section{Formal moduli problems and formal groupoids} \label{section formalmodulistuff}

  Much of the interesting geometry of a derived stack $\EuScript{X}$ is detectable at the infinitesimal level. Heuristically, a  derived stack can be viewed as a family of infinitesimal thickenings of a scheme, parametrized by its points. In the language of formal geometry, this can be made precise by the slogan that the map $X \to X_{dR}$ of a (nice enough) derived stack to its de Rham stack is a crystal of formal derived stacks. The reader may consult \cite[Section 2]{shiftedpoisson} for more on this perspective.

The sought-after relationship between the filtered loop space and the de Rham stack will exploit the fact that these two objects admit the structure of a \emph{formal moduli problem} (at least when $X$ is a scheme) relative to  $X$. A key feature of the theory of formal moduli problems, which we lift from \cite{gaitsgory2020study} is the well behaved correspondence between groupoid objects and formal moduli problems under $X$. We will eventually exploit this to relate the filtered loop space with the Hodge degeneration $X_{Hod}$.

As we follow the formalism developed in \cite{gaitsgory2020study}, we review some basic constructions and definitions found therein. In particular, we highlight the distinction between formal moduli problems \emph{over} a given stack and formal moduli problems \emph{under} a stack.  In order to proceed, we first recall what it means for a stack to admit a \emph{deformation theory}.  The basic setup here is over a field $k$ of characteristic zero.

\subsection{DAG preliminaries}

Before diving into some formal geometry, we review the notions of derived stacks which we will be working in.

We recall that there are two variants of ``derived" geometric objects , one whose affine objects are connective $E_{\infty}$-rings, and one where the affine objects are simplicial commutative rings.  In characteristic zero, the two contexts are equivalent.  We review parallel constructions from both simultaneously, as we will switch between both settings.  

Fix a commutative ring $R$ and let $\EuScript{C} = \{ \on{CAlg}^{\on{cn}}_R, \on{sCAlg}_R \}$ denote either of the $\infty$-category of connective $R$-algebras or the $\infty$-category of simplicial commutative $R$-algebras. Recall that the latter can be characterised as the completion via sifted colimits of the category of (discrete) free $R$-algebras. There exists a functor
$$
\theta: \on{sCAlg}_R \to \on{CAlg}_{R}^{\on{cn}};
$$
which takes the underlying connective $E_\infty$-algebra of a simplicial commutative algebra. This preserves limits and colimits so is in fact monadic and comonadic. In characteristic zero, this is in fact an equivalence, and this is often the setting we will find ourselves in within this paper.  

In any case one  may define a derived stack via its functor of points, as an object of the $\infty$-category 
$\on{Fun}(\EuScript{C}, \mathcal{S})$ satisfying hyperdescent with respect to a suitable topology on $\EuScript{C}^{op}$, e.g the \'{e}tale topology. Throughout the sequel we distinguish the context we  are working in by letting $\on{dStk}_{R}$ denote the $\infty$-category of derived stacks and let $\on{dStk}^{E_\infty}_R$ denote the $\infty$-category of ``spectral stacks" over $R$. In either cases, one obtains an $\infty$-topos, which is Cartesian closed, so that it makes sense to talk about internal mapping objects:  given any two $X,Y \in \on{Fun}(\EuScript{C}, \mathcal{S})$, one forms the mapping stack $\Map_{\EuScript{C}}(X, Y)$,  In various cases of interest, if the source and/or target is suitably representable by  a derived scheme or a derived Artin stack, then this is the case for $\Map_{\EuScript{C}}(X, Y)$ as well. 
\\

In addition we also will occasionally find ourselves working with \emph{higher stacks}. This genre of geometric objects was introduced by Simpson, and essentially is composed of stacks on the site of discrete commutative rings equipped with for example the \'{e}tale or fppf topology. For the sake of maintaining a somewhat self-contained exposition, we now define this as following:

\begin{defn}
Let $\on{Aff}_R$ be the category of affine schemes, equivalently the opposite of the category of discrete commutative $R$-algebras equipped with some (classical) Grothendieck topology. Then we set 
$$
\on{Stk}_R :=  \on{Shv}^{\tau}_{R}  := \on{Fun}(\on{CAlg}_R, \mathcal{S})^{\tau} 
$$
to be the $\infty$-category of higher stacks, equivalently that of sheaves of spaces with respect to the topology $\tau$.  
\end{defn}
 
\begin{rem}
Higher stacks provide the natural ambient setting for the notion of \emph{affine stacks}, introduced by Toën in \cite{toen2006champs} which we will briefly review in Section \ref{section affine cogroupoids}. The cogroupoid objects we construct and study in this Section \ref{section affine cogroupoids} will typically live in this setting of higher stacks. It is an interesting phenomenon, that this setting which is somehow ``discrete in the domain" and "derived" or ``homotopical" in the target, provides a home for affine objects which are, by their very nature, coconnective, at least when viewed as $E_\infty$-algebras. 
\end{rem}

\subsection{Formal moduli problems}
In this section we set up the necessary background regarding formal moduli problems. These objects capture the infinitesimal part of the geometry of derived stacks and feature prominently in this work. First we recall some auxiliary notions, leading up to the notion of a formal moduli problem.  

\begin{defn} \label{convergentproperty}
A stack $\EuScript{X}$ is \emph{convergent} or \emph{nil-complete} if for derived affine scheme $\Spec B$, the natural map
$$
 \EuScript{X}(B) \to \lim_k \EuScript{X}(\tau_{\leq k}B) 
$$
is an equivalence. 
\end{defn}

\begin{defn}(\cite[Section 0.1]{gaitsgory2020study})\label{deformationtheory}
Let $\EuScript{X}$ be a derived stack. We say that $\EuScript{X}$ admits a \emph{deformation theory} if it is convergent, and  such that for every pushout square of affine schemes 

$$
\xymatrix{
& S_1 \ar[d] \ar[r] & S_2 \ar[d] \\
 & S^{'}_{1}  \ar[r] & S^{'}_{2}, 
}
$$
where the map $S_1 \to S^{'}_{1}$ is a \emph{nilpotent embedding} (i.e. the map on truncations is a closed embedding with nilpotent ideal of definition), the resulting diagram   
$$
\xymatrix{
& \Map(S^{'}_{2}, \EuScript{X}) \ar[d] \ar[r] & \Map(S_{2}, \EuScript{X}) \ar[d] \\
 & \Map(S^{'}_{1}, \EuScript{X})   \ar[r] & \Map(S_{1}, \EuScript{X}) , 
}
$$
is a pullback diagram. 
\end{defn}

\begin{rem}
In fact this is enough, by \cite{gaitsgory2020study} to guarantee the existence of the \emph{pro-cotangent  complex} of $\EuScript{X}$. This is an assignment, for every derived scheme $x: S \to \EuScript{X}$,of a pro-object 
$$
T^*_{x}(\EuScript{X})   \in \on{Pro}(\QCoh(S)),
$$
which governs the infinitesimal behavior of $\EuScript{X}$ at the point $x$. 

\end{rem}
\begin{rem}
Any $n$-Artin stack, for example, satisfies the properties of Definition   \ref{deformationtheory}, and thus admits a deformation theory. 
\end{rem}

We also recall the notion of `locally almost of finite type"(=laft) from \cite{gaitsgory2019study}. For this we first need to recall what it means for an (derived) affine scheme to be locally of finite type. 

\begin{defn}
Let $X = \Spec A$ be a derived affine scheme. Then $X$ is of finite type if $\pi_0(A)$ is of finite type over $k$ and if each $\pi_n(A)$ is finitely generated as a module over $\pi_0(A)$. 
\end{defn}

Next one defines what it means for a (pre)stack to be locally of finite type. 

\begin{defn}
Let $\EuScript{X}$ be an ``$n$-coconnective" derived stack. We say that $\EuScript{X}$ is locally of finite type if it arises as the left Kan extension of its own restriction along the embedding
$$
\on{Sch}^{\on{aff}}_{\on{ft}} \hookrightarrow \on{Sch}^{\on{aff}}.
$$
In particular this means that $\EuScript{X}$ is locally of finite type if it is determined by its values on affine schemes of finite type. 
\end{defn}

Now we define what it means for a general stack to be ``laft".

\begin{defn} \label{laftnessforstacks} 
Let $\EuScript{X}$ be an arbitrary stack. Then we say that it is locally almost of finite type if the following conditions hold:
\begin{itemize}
    \item $\EuScript{X}$ is nil-complete (or convergent) (Definition \ref{convergentproperty})
    \item For every $n$, we have that  $^{\leq n} \EuScript{X}$ is locally of finite type. 
\end{itemize}
\end{defn}

\begin{rem}
A key reason for working with laft  stacks  (as well as the notion of an \emph{inf scheme} appearing in the following definition) is that this is the ``right" framework in order to set up the correct functoriality for $\on{IndCoh}(-)$. For the sake of completeness we have included the definitions, but we will not need to focus on this condition in any particular depth.  
\end{rem}
We are finally ready to define the main objects of this section, formal moduli problems. 

\begin{defn}\label{fmpoverastack}
Let $\EuScript{X}$ be a derived stack locally almost of finite type. The $\infty$-category of formal moduli problems \emph{over} $\EuScript{X}$ is the full subcategory spanned by $\EuScript{Y} \to \EuScript{X}$ for which the map is
\begin{itemize}
    \item inf-schematic. This means that the base change along a map $\Spec B \to \EuScript{X}$ is an inf-scheme, i.e it satisfies the laft condition, it admits a deformation theory and its reduction will be a reduced quasi-compact scheme. (cf. \cite[Chapter 2]{gaitsgory2020study}) 
    \item a nil-isomorphism. Recall that this means that the map of the ``reduced stacks"  $\EuScript{X}_{red} \to \EuScript{Y}_{red}$ is an isomorphism. 
\end{itemize}
We denote this category by $\on{FMP}_{/ \EuScript{X}}$.
\end{defn}
Next we define the notion of formal moduli problems under a fixed (pre)stack. 

\begin{defn}
Let $\EuScript{X} \in \on{dStk}_{\on{laft-def}}$ be a fixed (pre)stack which is both locally almost of finite type, and which admits a deformation theory.  The $\infty$-category of formal moduli problems \emph{under} $\EuScript{X}$ is spanned by those $\EuScript{X }\to \EuScript{Y}$ for which:
\begin{itemize}
    \item $\EuScript{Y}$ is itself locally of finite type and admits a deformation theory.
    \item the map $\EuScript{X} \to \EuScript{Y}$ is a nil-isomorphism.
\end{itemize}
We denote this $\infty$-category by $\on{FMP}_{\EuScript{X}/}$
\end{defn}

Next we define the notion of a formal groupoid: 

\begin{defn} \label{formalgroupoid}
Let $\EuScript{X}$ be a fixed derived stack, locally almost of finite type. We let $\on{FormGrpoid}(\EuScript{X})$ denote the $\infty$-category of groupoid objects in  $\on{FMP}_{/\EuScript{X}}$
\end{defn}
 
The key result which we borrow from this theory will be the following:

\begin{thm} \label{gaitsgorymain}(\cite[Theorem 2.3.2]{gaitsgory2020study}
Let $\EuScript{X}$ be a stack which admits a deformation theory. Then there is an equivalence of $\infty$-categories:
$$
B_{\EuScript{X}}: \on{FormGrpoid}(\EuScript{X}) \leftrightarrows \on{FormMod}_{\EuScript{X}/} : \EuScript{N}_{\EuScript{X}}(-)
$$
\end{thm}

In effect, this states that there is a well-defined procedure of taking the quotient by a formal groupoid to obtain a formal moduli problem under $X$. 

\begin{rem}
 It is sensible to study $E_2$-formal groupoids, and indeed, $E_n$-formal groupoids. As a formal consequence  of Theorem \ref{gaitsgorymain}, there exists an ``iterated" quotient or classifying stack of an $n$-formal groupoid over $\EuScript{X}$; this will give rise to a formal moduli problem under $\EuScript{X}$. 
\end{rem}

\begin{cor}
Let $\EuScript{X}$ be as above. Then there exists an equivalence 
$$
B_{\EuScript{X}}^{(n)}: \on{FormGrpoid}^{(n)}(\EuScript{X}) \leftrightarrows \on{FormMod}_{\EuScript{X}/} : \EuScript{N}^{(n)}_{\EuScript{X}}(-)
$$
Here the left hand side denotes the $\infty$-category of $E_n$-groupoid objects in $\on{FormGrpoid}_{/ \EuScript{X}}$. 
\end{cor}

\subsection{Formal completions} \label{section formal completions}
Given a morphism $\EuScript{X} \to \EuScript{Y}$ there exists a canonical formal moduli problem under $\EuScript{X}$; we may view this as an infinitesimal thickening on $\EuScript{X}$ in $\EuScript{Y}$.  We now briefly review this construction cf. \cite{gaitsgory2020study, calaque2017shifted}. 

\begin{const}
Let $\EuScript{X} \to \EuScript{Y}$ be a morphism of derived stacks. The formal completion of $\EuScript{Y}$ along $\EuScript{X}$ is defined to be the pullback 
$$
 \EuScript{Y}_{\EuScript{X}}^{\wedge} := \EuScript{Y} \times_{\EuScript{Y}_{dR}} \EuScript{X}_{dR} 
$$
The maps $\EuScript{X} \to \EuScript{Y}$ and $\EuScript{X} \to \EuScript{X}_{dR}$ induce a map $\EuScript{X} \to \EuScript{Y}_{\EuScript{X}}^{\wedge}$ which is a nil-isomorphism. Hence this procedure defines an object in $\on{FormMod}_{\EuScript{X}/}$
\end{const}

As a consequence of Theorem \ref{gaitsgorymain}, the canonical morphism $\EuScript{X} \to \EuScript{Y}_{\EuScript{X}}^{\wedge}$  to the infinitesimal thickening in $\EuScript{Y}$ may be recovered as the realization of its Čech nerve. The reader can compare this to the discussion in \cite[Section 4.2]{toensurvey} for a related universal property of the formal completion with respect to maps into arbitrary derived schemes.

\section{Cogroupoid objects} \label{section cogroupoid objects}

In this section we recall some basic facts about groupoid objects in $\infty$-categories, as well as the dual notion of a cogroupoid object, which will play a key role in our constructions.

\begin{defn}
Let $\EuScript{C}$ be an $\infty$-category and let $\on{Fun}(\Delta^{op}, \EuScript{C})$ be the $\infty$-category of simplicial objects in $\EuScript{C}$.  We say $X_\bullet \in \on{Fun}(\Delta^{op}, \EuScript{C})$ is a \emph{groupoid object} of $\EuScript{C}$ if, for every $n \geq 0$, and every partition $[n]= S \cup S'$ , such that $S \cap S'$ consists of a single element $s$, the diagram 
$$
\xymatrix{
& X([n]) \ar[d] \ar[r] &  X(S) \ar[d] \\
 & X(S')  \ar[r] & X(\{s\})
}
$$
is a pullback square in $\EuScript{C}$. 
\end{defn}

\begin{rem}
A group object in $\EuScript{C}$ is a groupoid object  $X_\bullet$  for which $X_0 \simeq *$.
\end{rem}

We have the following dual notion of a cogroupoid. 

\begin{defn}
Let $\EuScript{C}$ be as above, and let $\on{Fun}(\Delta, \EuScript{C})$ denote the $\infty$-category of cosimplicial objects of $\EuScript{C}$. We say $X^\bullet \in \on{Fun}(\Delta, \EuScript{C})$ is a cogroupoid object if it is a groupoid object in the opposite category $\EuScript{C}^{op}$. In particular, for every   partition $[n]= S \cup S'$ , such that $S \cap S'$ consists of a single element $s$, the diagram 
$$
\xymatrix{
& X(\{s\}) \ar[d] \ar[r] &  X(S) \ar[d] \\
 & X(S')  \ar[r] & X([n])
}
$$
is a pushout square in $\EuScript{C}$. We use the notation $\on{coGrpd}(\EuScript{C})$ to denote the $\infty$-category of cogroupoid objects in $\EuScript{C}$. 
\end{defn}

\begin{rem}
A cogroup object in $\EuScript{C}$ is a cogroupoid object  $X_\bullet$  for which $X_0 \simeq \emptyset$.
\end{rem}

\begin{ex}
Let $\EuScript{C} = $ denote the $\infty$-category of pointed spaces. For any pointed space $X \in \mathcal{S}_*$, its suspension  $\Sigma X$ is canonically a cogroup object. To see this, let  $ \pi: X \to *$ be the  map from $X$ to the final object. Then the conerve, $\on{coNerve}(\pi)$, 
$$
* \rightrightarrows * \bigsqcup_{X} * \substack{\rightarrow\\[-1em] \rightarrow \\[-1em] \rightarrow} * \bigsqcup_{X} * \bigsqcup_{X} * ...
$$
precisely packages $\Sigma X \simeq   * \sqcup_{X} * $ together with its cogroup structure maps. Setting $X= S^0$ recovers the $E_1$ cogroup structure on $S^1$ in pointed spaces. 
\end{ex}

\begin{const} \label{conerveconstruction}
The above example is an instance of the conerve construction, which we now describe. Let $\EuScript{C}$ be an $\infty$-category which admits finite colimits. Let $f: Y \to X $  be a morphism in $\EuScript{C}$. Then we define the conerve of $f$, $\on{coNerve}(f)$ to be the cosimplicial object 
$$
\on{coNerve}(f): = X \rightrightarrows X \bigsqcup_{Y} X \substack{\rightarrow\\[-1em] \rightarrow \\[-1em] \rightarrow} X \bigsqcup_{Y} X \bigsqcup_{Y} X ...
$$
\end{const}

\noindent 
One can iteratively define the notion of an $E_n$-cogroupoid:

\begin{defn}
We define an $E_2$-cogroupoid object to be a cogroupoid object in the $\infty$-category $\on{coGrpd}(\EuScript{C})$. Proceeding iteratively, we define an $E_n$-cogroupoid to be an a cogroupoid object in $\on{coGrpd}_{n-1}(\EuScript{C})$. 
\end{defn}

\begin{ex}
Let $\phi: \emptyset \to \pt$ be the map in spaces from the initial object to the final object. Then we set
$$
S^{0, \bullet} := \on{coNerve}(\phi)
$$
Now let $\phi^\bullet: S^{0, \bullet} \to *$ be the map to the final cogroupoid object in spaces. We set 
$$
S^{1, \bullet, \bullet} := \on{coNerve}(\phi^\bullet)
$$
Of course, one need not stop here; for each $n$ there is an $E_{n+1}$ cogroupoid  with the $n$-sphere as the degree $(1,...,1)$ space of morphisms. 
\end{ex}

\section{Affine cogroupoids over $\filstack$} \label{section affine cogroupoids}

Let $R$ be a commutative ring. There exists a distinguished class of (higher) stacks over $R$ which are completely determined by their cohomology (together with its additional yet canonical structure) as a cosimplicial commutative algebra. The study of this class of stacks, known as \emph{affine stacks} was initiated in \cite{toen2006champs}. 

The purpose of this section is to rephrase these constructions internally to the filtered setting, i.e. over the stack $\filstack$. 

\subsection{Filtrations and $\filstack$}
We would like to remind the reader of some basic notions from \cite{geometryoffiltr}, motivating our extensive usage of the stack $\filstack$. 

\begin{defn}
We define the $\infty$-category of filtered $R$-modules to be  
$$
\on{Fil}(\Mod_R) = \on{Fun}(\Z^{op}, \Mod_R), 
$$
where $\Z$ is to be viewed as a poset. Similarly we define the $\infty$-category 
$$
\on{Gr}(\Mod_R) = \on{Fun}(\Z^{ds, op}, \Mod_R)
$$
to be the $\infty$-category of graded $R$-modules, where $\Z^{ds}$ denotes the integers viewed as a discrete space. These both obtain a symmetric monoidal structure, given by Day convolution. 
\end{defn}

\begin{const}
There exist symmetric monoidal functors 
$$
\on{Und}: \on{Fil}(\Mod_R) \to \Mod_{R},
$$
taking a filtered $R$-module to its underlying object, and 
$$
\gr: \on{Fil}(\Mod_R) \to \on{Gr}(\Mod_R)
$$
\end{const}

Now we turn to an algebro-geometric incarnation of these notions. Let $\filstack$ be the quotient stack of $\AAA^1$ by the canonical $\GG_m$ action by dilation. This stack comes equipped with two distinguished maps:
$$
  \Spec k \xrightarrow{\eta} \filstack \xleftarrow{\iota} B \GG_m
$$
which throughout this paper, we refer to as the \emph{generic point} and \emph{central point} respectively. 
Viewing it as a derived stack, one obtains the following relationship with filtered objects at the level of quasi-coherent sheaves:

\begin{thm}[cf. \cite{geometryoffiltr}]
There is a symmetric monoidal equivalence of stable $\infty$-categories
$$
\QCoh(\filstack) \simeq \on{Fil}(\Mod_R)
$$
Under the above equivalence, the associated graded functor  $\on{gr}: \on{Fil}(\Mod_R) \to \on{Gr}(\Mod_R)$ is naturally identified with the pullback
$$
\QCoh(\filstack) \to \QCoh(B \GG_m) \simeq \on{Gr}(\Mod_R); 
$$
the functor $\on{Und}: \Sp^{fil} \to \on{Sp}$ sending a filtered $R$-module to the $R$-module underlying the filtration is naturally identified with the pullback
$$
\QCoh(\filstack) \to \QCoh(\Spec R) \simeq \Mod_R .
$$
\end{thm}

\begin{rem}
Given the above equivalence, we may view (derived) stacks admitting a map  
$$
\EuScript{X} \to \filstack 
$$
in two ways. First, we may think of them as degenerations, from the \emph{generic fiber}  $\EuScript{X}_1 := \eta^*(\EuScript{X})$  to the \emph{special fiber} $\EuScript{X}_0:= \iota^*(\EuScript{X})$. We may also think of them as filtrations on the cohomology $R\Gamma(\EuScript{X}_{1}, \OO_{\EuScript{X}_1})$ of the generic fiber. In this paper, we straddle the line between these two perspectives. 
\end{rem}
\subsection{Affine stacks over $\filstack$} 
\begin{defn}
Let $\QCoh(\filstack)^{\heartsuit}$ denote the heart of the stable $\infty$-category of quasi-coherent sheaves. By \cite[Section 8]{geometryoffiltr}, this abelian category is equivalent to the filtered diagrams of objects belonging to the heart of $\Mod_R$. Let $\on{CAlg}(\QCoh(\filstack))^{\heartsuit}$ denote the category of algebras in  $\QCoh(\filstack)^{\heartsuit}$.  We set 
$$
\mathsf{coSCR}_{\filstack} := \mathsf{coSCR}(\on{CAlg}(\QCoh(\filstack)^{\heartsuit}))
$$
to be the underlying $\infty$-category of the category of cosimplicial objects in $\on{CAlg}((\QCoh(\filstack)^{\heartsuit}))$. 
\end{defn}

\begin{rem}
There is an alternative way to understand the  $\infty$-category $\mathsf{coSCR}_{\filstack}$. Let \\
$\QCoh(\AAA^1)^\heartsuit \simeq \Mod_{k[t]}$ denote the classical abelian category of $R[t]$-modules. The canonical inclusion $\GG_m \to \AAA^1$ gives an action of $\GG_m$ on $\AAA^1$, and thus we may study $\GG_m$-equivariant objects in the above abelian category. This forms an abelian category, which we denote by $\on{Rep}_{\AAA^1}(\GG_m)$
and we can canonically form the category of cosimplicial modules in this category, $\on{Rep}_{\AAA^1}(\GG_m)^{\Delta}$. This admits a symmetric monoidal structure (the point-wise one) and so we may take commutative monoids here. By the arguments of \cite[Section 3.4]{katzarkov2009algebraic}, this forms a model category and we can define $\mathsf{coSCR}_{\filstack}$ to be its underlying $\infty$-category.
\end{rem}

\begin{rem}
Let $ \Mod^{\on{ccn}}_{R}$ denote the full subcategory of $\Mod_R$  consisting of modules $M$ for which $\pi_n(M) =0$ if $n >0$. We remark that there exists a monad $T$ on $\Mod^{\on{ccn}}_{R}$ for which there is an equivalence 
$$
\Alg_{T} \simeq \mathsf{coSCR}_R
$$
This can be extracted from \cite[Proof of Théorème 2.1.2]{toen2006champs}.
When $R$ is a $\mathbb{Q}$-algebra this agrees with the free $E_\infty$-algebra functor, so that in fact we have an equivalence 
$$
\mathsf{coSCR}_R \simeq \on{CAlg}^{\on{ccn}}_{R};
$$
thus one may think of cosimplicial commutative algebras as coconnective $E_\infty$-algebras in this context. Similarly, when working over $\filstack$ in characteristic zero, there will be an equivalence 
$$
\mathsf{coSCR}_{\filstack} \simeq \on{CAlg}^{\on{ccn}}( \QCoh(\filstack))
$$
\end{rem}

\begin{const}

Let $\Spec_{\filstack}: \on{CAlg}(\QCoh(\filstack)^{\heartsuit})^{op} \to \on{Stk}_{\filstack}$, denote the relative spectrum functor, sending a commutative algebra in $\QCoh(\filstack)^{\heartsuit}$ to a relative affine (underived) scheme over $\filstack$. We form the  Kan-extension of this functor along the inclusion 
$$
\on{CAlg}(\QCoh(\filstack)^{\heartsuit}) \to \mathsf{coSCR}(\QCoh(\filstack)^{\heartsuit})
$$ to obtain a functor 
$$
\Spec_{\filstack}^{\Delta}: \mathsf{coSCR}_{\filstack}^{op} \to \on{Stk}_{\filstack}. 
$$
\end{const}

\begin{rem}
This possesses a left adjoint, namely the global sections functor 
$$
\OO: \on{Stk}_{\filstack} \to \mathsf{coSCR}_{\filstack}^{op}.
$$
To understand this functor, we remark that one can work with a point-set model of $\on{Stk}_{\filstack}$ where every object $\EuScript{F}$ in $\on{Stk}_{\filstack}$ has a a model as a simplicial object $\{\EuScript{F}_\bullet\}$ in presheaves (of sets) over the stack $\filstack$. Thus we define the cosimplicial algebra  $\OO(\EuScript{F})^\bullet$
$$
\OO(\EuScript{F})^{n} := \OO(\EuScript{F}_n) 
$$
As in \cite[Section 2.2]{toen2006champs}, this functor is left Quillen and thus induces a functor at the level of $\infty$-categories. 
\end{rem}

\begin{prop}
The functor 
$$
\Spec^{\Delta}_{\filstack}: \mathsf{coSCR}_{\filstack}^{op} \to \on{Stk}_{\filstack}
$$
is fully faithful. 
\end{prop}

\begin{proof}
Note that we may write any cosimplicial algebra $A^{\bullet}$ as the cosifted limit of objects in $\on{CAlg}_R$, i.e. of discrete $R$-algebras.  Then it follows that
$$
A^{\bullet} \simeq \lim_{\Delta} A^n \simeq \OO (\colim_{\Delta^{op}}[\Spec^{\Delta}_{\filstack}(A_n )]) \simeq \OO \Spec^{\Delta}_{\filstack}(A^{\bullet})
$$
is an equivalence. 
\end{proof}

The following proposition, motivated by \cite[Th\'{e}or\`{e}me 2.29, Corollaire 2.2.10]{toen2006champs}, says that modulo potential size issues, the $\infty$-category of (filtered) affine stacks behaves like a localization of the $\infty$-category of higher stacks. 

\begin{prop}
Let $\on{Stk}_{\filstack}$ denote the subcategory of stacks such that $\OO(X)$ is $\mathbb{U}$-small with respect to a fixed universe $\mathbb{U}$. Then the functor 
$$
X \mapsto \on{Aff}(X)= \Spec_{\filstack}^{\Delta} \OO(X) 
$$ 
is a localization. Moreover, $X$ is an affine stack over $\filstack$ if and only if the adjunction morphism
$$
X \to \on{Aff}(X)  
$$
is an equivalence in $\on{Stk}_{\filstack}$.
\end{prop}

\begin{proof}
The arguments for the proof of \cite[Théorème 2.29]{toen2006champs} work verbatim when working with  $\mathsf{coSCR}_{\filstack}$ instead of $\mathsf{coSCR}_{R}$.

\end{proof}
\section{Filtered $n$-spheres and (higher) loop spaces} \label{cogroupoidcircle}
In \cite{moulinos2019universal}, the authors constructed the  \emph{filtered circle}, as an affine abelian group stack over $\filstack$. This roughly packages the data arising from the Postnikov filtration on the cochain complex on $S^1$, together with its compatibility with the group strucure on $S^1$. We will now see that it fits into the larger structure of an $E_2$-cogroupoid object in affine stacks.   

\begin{defn}
Following the discussion in  \cite[Section 1.3]{toen2020classes} we define  the ``quantum point" as the stack $\mathcal{Q}:= B \GG_m$, viewed as an object over $\filstack$ via the morphism $\iota: B \GG_m \to \filstack$. 
This will have generic fiber the ``null scheme"
$$
\mathcal{Q}_1 = \mathcal{Q} \times_{\filstack} \Spec R = \emptyset 
$$
and as central fiber the derived scheme  
$$
\mathcal{Q}_0 = \mathcal{Q} \times_{\filstack} B \GG_m = \Spec(R \oplus R[1](-1))
$$
\end{defn} 

\begin{rem}
This can alternatively be described as $\Spec_{\filstack}(\AAA)$ (this denotes the relative spectrum) where $\AAA$ is the image of the unit $\mathbf{1} \in \on{Gr}_R \simeq \QCoh(B \GG_m)$ in $ \QCoh(\filstack) \simeq \on{Fil}_R$ along the pushforward functor  
$$
\iota_*: \QCoh(B \GG_m) \to \QCoh(\filstack)
$$
This is a lax symmetric monoidal (in fact, it ends up being symmetric monoidal), so that $\AAA$ acquires an $E_\infty$-algebra structure.  
Alternatively, one may describe this functor as the symmetric monoidal functor 
$I: \on{Gr}_R \to \on{Fil}_R$ 
given by left Kan extension along $\Z^{ds} \hookrightarrow \Z$. 
\end{rem}

We now use $\mathcal{Q}$ to construct a cogroupoid object in the $\infty$-topos  $\on{Stk}_{\filstack}$.

\begin{const} \label{construction of filtered zero sphere}
Let $\iota: B \GG_m \simeq \mathcal{Q} \to \filstack$. Let 
$$
\phi: \OO_{\filstack} \to \iota_*(\OO_{B \GG_m}) 
$$ 
be the unit map of commutative algebra objects in $$\on{CAlg}(\QCoh(\filstack)) \simeq \on{CAlg}(\on{Fil}(\Mod_R)).
$$ Finally, let $N(\phi)_\bullet$ be the nerve of this map, viewed as a simplicial object in this $\infty$-category; by construction this will be a groupoid object in $\on{CAlg}(\on{Fil}(\Mod_R))$. We note that this is levelwise discrete, in the homotopical sense. We define 
$$
S^{0, \bullet}_{fil} = \Spec(N(\phi)_\bullet);
$$
this will be a cogroupoid object in the $\infty$-category of derived affine schemes over $\filstack$.  
\end{const}

\begin{rem}
In simplicial degree $1$, one recovers the filtered stack (cf. \cite[Section 5.1]{filteredformal}) $S^0_{fil}$, which we refer to as the \emph{filtered zero sphere}. As is described in loc. cit., 
one may express the fiber product  $\OO_{\filstack} \times_{\iota_*(\OO_{B \GG_m})} \OO_{\filstack}$ in terms of the equivalence of $\QCoh(\filstack)$ with $k[t]$-modules in graded complexes as  the discrete ring 
$$
k[t_1, t_2]/(t_1 +t_2)(t_1 - t_2).
$$
Since the $\eta^* \iota_*(\OO_{B \GG_m}) \simeq 0$,  one has equivalences 
$$
S^0_{fil}|_{\Spec(k)} \simeq S^0 \simeq \Spec(k) \sqcup \Spec(k)
$$
and 
$$
S^0_{fil}|_{B \GG_m} \simeq \Spec(k[\epsilon]/(\epsilon^2))
$$
where in the latter equivalence, we view  $k[\epsilon]/(\epsilon^2)$ as a graded commutative ring with $\epsilon$ in weight $-1$. 
\end{rem}

\begin{rem}
One may try to define this cogroupoid directly as the conerve of the map $\iota: \mathcal{Q} \to \filstack$, which we write as
$\on{coNerve}^{\bullet}(\iota ) $. We remark that this is related by affinization in the sense that there is an equivalence
$$
S^{0,\bullet}_{fil} \simeq \on{Aff}_{\filstack}(\on{coNerve}^{\bullet}(\iota ) ) 
$$
Note that in degree 1 of this construction, we obtain the suspension $\Sigma \mathcal{Q}$ in the $\infty$-category of derived stacks over $\filstack$ whereas for $S_{fil}^{0, \bullet}$, we obtain the suspension  $\Sigma_{aff} \mathcal{Q}$ in the full subcategory of \emph{affine stacks}. These are not in general equivalent. Indeed, by \cite{toen2006champs}, the inclusion of affine stacks into all stacks is a right adjoint functor, which typically does not preserve colimits.      
\end{rem}

\begin{const}
Given any derived stack $X \in \on{dStk}_{\filstack}$, we may form the levelwise mapping stack
$$
\Map_{\filstack}(S^{0, \bullet}_{fil}, X);
$$
this is a simplicial object in $\on{dStk}_{\filstack}$ which by construction of $S^{0, \bullet}_{fil}$, admits the structure of a groupoid object in derived stacks over $\filstack$. Note that we are working in an $\infty$-topos, where every groupoid object $\mathcal{G}_\bullet $ arises as the Čech nerve of an effective epimorphism. Hence, we may form the associated classifying stack $B \mathcal{G}_\bullet$, for which $\mathcal{G}_0 \to B \mathcal{G}_\bullet$ is an effective epimorphism. 
\end{const}

\subsection{The filtered circle as  $E_2$-cogroupoid}
One can give a construction of the filtered circle of \cite{moulinos2019universal} as a $2$-cogroupoid in the $\infty$-category of  affine stacks. This is done by taking the iterated nerve of the map $\phi: \OO_{\filstack} \to 0_*(\OO_{B \GG_m})$ in cosimplicial commutative algebra objects in $\QCoh(\filstack)$. 

\begin{const} \label{filteredcirclecogroupoid}
Let $N(\phi)_\bullet$ be as above. We define the $2$-fold iterated nerve of $\phi$ to be the bi-simplicial object $N(\phi)_\bullet$, defined to be the nerve of the map of simplicial objects 
$$
\OO_{\filstack} \to N(\phi)_\bullet,
$$
where the left-hand side is considered as a constant simplicial object. 
This will canonically give rise to a bisimplicial object in cosimplicial algebras which we denote by $N(\phi)_{\bullet, \bullet}$.  This gives rise to a 2-groupoid object in the $\infty$-category $\mathsf{coSCR}_R$. We can alternatively restrict to the diagonal simplicial object, at this stage. In any case we set 
$$
S^{1,\bullet, \bullet}_{fil}:=\Spec^{\Delta}_{\filstack}(N(\phi)_{\bullet, \bullet})
$$
This will be a 2-cogroupoid in the $\infty$-category of affine stacks over $\filstack$. 
\end{const}

\begin{const} \label{version2}
We introduce a variant of the above. 
Let $S^{0, \bullet}_{fil}$ be as above, and let 
$$
\phi: S^{0, \bullet}_{fil} \to *
$$
be a map in the $\infty$-category of cosimplicial objects in derived stacks. Here, the right hand side is taken to be the constant cosimplicial object which is $\filstack$ in each degree. 
$$
\Sigma S^{0}_{fil} :=\on{coNerve}(\phi)
$$
denote the conerve of this map, cf. construction \ref{conerveconstruction}
This will be a bicosimplicial object, and by the discussion above, a cogroupoid object in $\on{Stk}_R$. 
\end{const}

\begin{prop}
The cogroupoid object of Construction \ref{filteredcirclecogroupoid} is the affinization of $\Sigma S^{0}_{fil}$, i.e. there exists an equivalence 
$$
 S^{1, \bullet, \bullet}_{fil} \simeq \on{Aff}_{\filstack}(\Sigma S^{0}_{fil})
$$
\end{prop}

\begin{proof}
This is immediate upon noticing that there is a natural map 
$$
\Sigma S^{0}_{fil} \to S^{1, \bullet, \bullet}_{fil}
$$ 
which is an equivalence on cohomology. 
Indeed, one obtains the $E_2$-groupoid object of Construction \ref{filteredcirclecogroupoid} by applying the global sections $\OO(-)$ functor to Construction \ref{version2}. Applying then $\Spec^{\Delta}_{\filstack}(-)$  and noticing that we obtain the canonical affinization morphism,  the unit of the adjunction $\Spec^{\Delta}_{\filstack} \dashv  \OO(-)$. This will be an equivalence on cohomologies. 
\end{proof}

\noindent 
Via this description, we see that the degree $(1,1)$ piece of the $E_2$-cogroupoid  $S^{1, \bullet, \bullet}_{fil}$ is none other than the filtered circle of \cite{moulinos2019universal}. 

\begin{cor} \label{recoveryoffilteredcircle}
There is an equivalence: 
$$
S^1_{fil} \simeq S^{1,1,1}_{fil}
$$
of affine stacks. 
\end{cor}

\begin{rem}
Putting this together with Construction \ref{construction of filtered zero sphere}, we have the following identification
$$
S^1_{fil} \simeq \Sigma^{2}_{aff}(\mathcal{Q}) 
$$
exhibiting the filtered circle as a 2-fold suspension in affine stacks of $\mathcal{Q}$.  
\end{rem}

\begin{rem}
The reason we want to work with $\Aff_{\filstack}(\Sigma^2 \mathcal{Q})$ here as opposed to just $\Sigma^{2} \mathcal{Q}$ is that we would like to to recover the cogroupoid structure on the \emph{affine stack} $S^1_{Fil}$ of \cite{moulinos2019universal}.  Indeed the filtered loop space $\mathcal{L}_{fil}(X)$ was defined in loc. cit. as the stack of maps out of $S^1_{Fil}$. Thus we would like to study groupoid structures on the filtered loop space itself, as opposed to studying them on some object which is merely equivalent to it.  

\end{rem}
\begin{rem}
We remark that while one may recover the affine stack $S^1_{fil}$ in this way, this does not capture its structure as an abelian group stack. Indeed, this was studied in \cite{moulinos2019universal} by exhibiting  it as the classifying stack of a filtered abelian group scheme $\mathbb{H}$ over $\filstack$, which interpolates between the kernel and fixed points of the Frobenius on the Witt vector ring scheme $\mathbb{W}(-)$.
\end{rem}

\begin{const}
One can in fact iterate all of the above, eg. by taking iterated nerves of the map $\OO_{\filstack} \to 0_* B \GG_m$  in $\mathsf{coSCR}_{\filstack}$ and applying the $Spec^{\Delta}_{\filstack}$ functor to obtain  $S^n_{fil}$. We can summarize this discussion by saying that 
$$
S^n_{fil}:= \Spec^{\Delta}_{\filstack}(N^{\bullet, ... , \bullet}(\phi))
$$
acquires the structure of an $E_{n+1}$-cogroupoid  object in affine stacks over $\filstack$. 
\end{const}

\subsection{Filtered loop spaces}
Let $X$ be a derived scheme in characteristic zero. We describe the formation of higher filtered loop spaces, but focus in particular on the $n=1$ case. 

\begin{const}
Let $X$ be a derived scheme. Letting $X|_{\filstack} := X \times \filstack$, we set  
$$
\mathcal{L}^{(n)}_{fil}(X)=\Map(S^n_{Fil}, X|_{\filstack}),
$$
and refer to this as the  higher $n$-dimensional filtered loop space.
\end{const}

We would like to deduce that 

\begin{prop} \label{groupoidalnature}
Let $X$ be a derived affine scheme. Then the affinization morphism 
$$
\Sigma^{2} \mathcal{Q}^{\bullet, \bullet} \to \circlegroupoid
$$
induces an (levelwise)-equivalence of bi-cosimplicial derived stacks
$$
\Map_{\filstack}(\circlegroupoid,  X|_{\filstack}  ) \to \Map_{\filstack}(\Sigma^{2} \mathcal{Q}^{\bullet,\bullet} ,  X|_{\filstack}  ) 
$$
This makes the bi-cosimplicial object $\mathcal{L}_{fil}^{\bullet, \bullet}(X)$ into an $E_2$-groupoid in derived stacks. 
\end{prop}

\begin{proof}
For an arbitrary bidegree $(m,n)$ we have a canonical map 
\begin{equation} \label{levelwisequivalence}
    \Map_{\filstack}(S^{1, m, n}_{fil},  X|_{\filstack}  ) \to \Map_{\filstack}(\Sigma^{2} \mathcal{Q}^{m,n} ,  X|_{\filstack}  ) 
\end{equation}
induced by restriction along $\Sigma^{2} \mathcal{Q}^{m, n} \to S^{1,m,n}_{fil}$,  the levelwise affinization morphism. We would like to show that this is an equivalence. Granting this, we use the fact that $\Sigma^{2} \mathcal{Q}^{m,n}$ is a 2-cogroupoid in derived stacks, so that one gets a suitable product decomposition in (in each degree $(m,n)$ of te associated mapping stacks) satisfying the Segal conditions for a groupoid object. 

In order to display the map (\ref{levelwisequivalence}) as an equivalence, it will be enough to do so upon pulling back along $\eta, \iota$ respectively. We first treat the generic fiber case, i.e. the pullback along $\eta$. As we are in the characteristic zero setting, where simplicial and cosimplicial algebras can be modeled by connective and coconnective $E_\infty$ algebras respectively, we may appeal to \cite[Section 4]{lurie2011dag} , (see also \cite{ben2012loop}). In particular, the universal property in this setting states that the affinization functor is a monad in derived stacks; in particular, any map to a derived affine scheme out of a derived stack will factor through its affinization. In particular this gives an equivalence at the level of derived mapping stacks. 

For the same reason, one obtains an equivalence upon pulling back along $B \GG_m \xrightarrow{\iota} \filstack$. We can now conclude that the original map (\ref{levelwisequivalence}) defined over $\filstack$ is a levelwise equivalence. Hence, $\mathcal{L}^{\bullet, \bullet}_{fil}(X)$ is an $E_2$-cogroupoid in derived stacks. 
\end{proof}

\begin{rem}
The argument above uses the fact that we are in characteristic zero and the resulting notion of affinization in the derived setting. In positive or mixed characteristics, one can argue more generally  using the cohomological finite dimension of the stacks $\Sigma^{2} \mathcal{Q}^{m,n}$ and  $S^{1, m, n}_{fil}$ to obtain equivalence (\ref{levelwisequivalence}).

\end{rem}
\begin{rem}
While we have focused on the $n=1$ case, the considerations here are valid for higher loop spaces as well. Thus, one obtains a $E_{n+1}$ groupoid objects where the degree $(1,...1)$ object of morphisms is the $\Map_{\filstack}(S^n_{fil}, X|_{\filstack})$.
\end{rem}

\section{Hodge degeneration and deformation to the normal bundle}
In this section, we work over a characteristic zero base field $k$ and study the well known Hodge degeneration in its geometric form studied by Simpson (cf. \cite{simpson1996hodge, simpson1990nonabelian}  One avatar of the Hodge filtration on de Rham cohomology is given by the deformation to the normal cone construction. This manifests itself as a derived stack over $\filstack$ which specializes upon taking generic and central fibers to the de Rham stack and the Dolbeault stack respectively. 

\subsection{The de Rham and Dolbeault spaces}
We give a brief review here of the de Rham and Dolbeault space constructions in derived algebraic geometry. The results in this section are quite classical and well-known; we give them to motivate the objects studied. 

\begin{defn}
Let $X$ be a scheme. Then the de Rham space of $X$ is defined to be the moduli problem $X_{dR} \in \on{dStk}_k$ given by  
$$
X_{dR}(A) = X(\pi_0(A)/I),
$$
where $I$ denotes the nilradical of $\pi_0(A)$. 
\end{defn}

The de Rham space was originally defined by Simpson as the quotient stack of a certain groupoid object in formal schemes.  This  goes back to the original construction of Simpson, see, for example \cite{simpson1996hodge}. 

\begin{const}
Let $X$ be a smooth scheme and let $\pi: X \to X_{dR}$ denote the canonical map. One sees that this is an effective epimorphism in the $\infty$-category of derived stacks, and so it is the effective quotient of the its nerve groupoid. 
Thus we can define the de Rham stack as the realization of the following groupoid
$$
X \leftleftarrows (X \times X)\widehat{_{\Delta}} 
$$
where the object $(X \times X)\widehat{_{\Delta}}$ denotes the formal completion of $X \times X$ along the diagonal morphism $\Delta: X \to X$.
\end{const}

 A proof of the following statement, using ideas which go back to Grothendieck's work in the setting of infinitesimal cohomology, may be found in \cite{lurie2009notes}:

\begin{thm} \label{theorem classic D module}
Let $X$ be a smooth scheme over $k$. Then there is an equivalence of categories:
$$
\QCoh(X_{\on{dR}}) \simeq \mathcal{D}_{X}-\Mod
$$
\end{thm}

\noindent Thus, the de Rham stack gives a conceptual way to understand $\mathcal{D}_{X}$-modules on $X$. More generally, e.g. for a derived scheme, a Koszul dual variant of the left hand-side parametrizes the notion of \emph{crystals} on the infinitesimal site.  
\\

We now move on to the Dolbeault space. This construction is (again) originally due to Carlos Simpson as a natural geometric parametrization for the notion of Higgs bundles. First we recall the notion of tangent bundle in this setting:

\begin{defn}
Let $X$ be a derived scheme.  The tangent bundle of $X$ is defined to be 
$$
\mathcal{T}X= \Spec_{X} \Sym(\mathbb{L}_{X});
$$
thus it is a relative derived scheme over $X$. 
\end{defn}

\begin{rem}
The tangent bundle admits the structure of an abelian group object over $X$, which we may think of as a simplicial object 
$$
(\mathcal{T}X)_\bullet
$$
with 
$$
\mathcal{T}X_0 = X  \, \, \, \, \,  \mathcal{T}X_2 = \mathcal{T}X \times_{X} \mathcal{T}X 
$$
\end{rem}

\begin{const} 
One can take the formal completion along the unit section 
$$
u: X \to \mathcal{T}X. 
$$
Given the machinery established thus far, we define this to be the pullback 
$$
\widehat{\mathcal{T}X} = X_{dR} \times_{{\mathcal{T}X}_{dR}}  \mathcal{T}X
$$
In fact, one may apply this to the map $X \to \mathcal{T}^n X$ to obtain a simplicial object in formal moduli problems, $\widehat{\mathcal{T}X}_{\bullet}$.  We remark that this becomes a groupoid object in formal moduli problems, cf. \cite{gaitsgory2020study}. 
\end{const}

\begin{defn}
We define, following Simpson, the Dolbeault stack to be the delooping
$$
X_{Dol}:= B_{X}\widehat{\mathcal{T}X_\bullet}
$$
of the formal group object $\mathcal{T}X_\bullet$ 
\end{defn}

\begin{prop}
There is an equivalence of stable $\infty$-categories
$$
\QCoh(X_{Dol}) \simeq \Mod_{\Sym({\mathbb{T}}X)}
$$
where $\mathbb{T}_X$ denotes the $\OO_X$-linear dual of the cotangent complex $\mathbb{L}_X$. 
\end{prop}

\begin{proof}
An application of the comonadic form of the Barr-Beck theorem shows that 
$$
\QCoh(X_{Dol}) \simeq \on{coMod}_{\widehat{\on{Sym}(\mathbb{L}_X)}}(\QCoh(X)).
$$
Alternatively, this can be extracted as the descent data defining $X_{Dol}$ as the realization of the groupoid  $\mathcal{T}X_{\bullet}$ corresponding to the formal group $\widehat{\mathcal{T}X}$ over $X$.
We now take the $\OO_X$-linear dual of the action map, for any fixed $M \in \QCoh(X_{Dol})$
$$
M \to M \otimes_{\OO_X}\widehat{\on{Sym}(\mathbb{L}_X)} 
$$
to obtain a map 
$$
(\widehat{\on{Sym}(\mathbb{L}_X)})^{\vee} \otimes_{\OO_X} M \to M,
$$
giving $M$ the structure of a $\widehat{\on{Sym}(\mathbb{L}_X)})^{\vee}$-module. It thus amounts to identify $\widehat{\on{Sym}(\mathbb{L}_X)})^{\vee}$ with $\Sym(\mathbb{T}X)$. For this, note first that 
$$
\widehat{\on{Sym}(\mathbb{L}_X)} \simeq \prod \Sym^{n}(\mathbb{L}_X), 
$$
so that 
$$
(\widehat{\on{Sym}(\mathbb{L}_X)})^{\vee} \simeq \bigoplus \Sym^{n}(\mathbb{L}_X)^{\vee}
$$
To identify each summand, we use the fact that $(\mathbb{L}_X)^{\vee} \simeq \mathbb{T}_X $and that we are in characteristic zero (so that divided and symmetric powers coincide) to conclude that 
$$
\Sym^{n}(\mathbb{L}_X)^{\vee} \simeq \Sym^{n}(\mathbb{T}_X)
$$
for every $n$. Thus we deduce the equivalence
$$
\widehat{\on{Sym}(\mathbb{L}_X)})^{\vee}\simeq \on{Sym}(\mathbb{T}_X).
$$
From this, we may conclude that $M$ is naturally a $\on{Sym}(\mathbb{T}_X)$-module. 
\end{proof}

As a consequence of the above, one may think of $\QCoh(X_{Dol})$ as a derived $\infty$-category of Higgs sheaves, cf. \cite{simpson1990nonabelian}.

\subsection{$\lambda$-connections and the deformation to the normal cone }

There is a natural $1$-parameter deformation of the notion of a flat connection. This is the notion of a $\lambda$-connection, going back to Deligne.

\begin{defn}
Let $E$ be a vector bundle on a scheme $X$. A $\lambda$-connection is an operator 
$$
\nabla_{\lambda}: E \to E \otimes \Omega^1_{X}
$$
for which 
$$
\nabla_{\lambda}(ae) = \lambda d(a)e + a \nabla(e)
$$
\end{defn}

\begin{rem}
If $\lambda \in R^\times$, then a $\lambda$ connection $\nabla_{\lambda}$  on a vector bundle $E$ gives equivalent data to a connection $\lambda^{-1} \nabla_{\lambda}$ On the other hand, if $\lambda= 0$, this is precisely the data of a Higgs bundle.    
\end{rem}

Let us give another perspective for the notion of $\lambda$-connection. Recall that the ring of differential operators $\mathcal{D}_X$ comes equipped with an filtration known as the \emph{order filtration}. This has associated graded $\OO_{\on{T}^*X} \simeq \Sym(\mathbb{T}_X)$. Applying the Rees construction to this gives the following $\GG_m$-equivariant sheaf of algebras over $X \times \AAA^1$.

\begin{defn}
    Let $\mathcal{D}^{\lambda}_{X}$ be the sheaf of algebras  on $X \times \filstack$ defined by 
    $$
    \mathcal{D}^{\lambda}_{X} = \bigoplus _{k \geq 0} t^k \mathcal{D}^{\leq k} \subset \mathcal{D}_{X} \otimes R[t],
    $$
    where the coordinate $t$ acts as $\lambda$. There is an evident $\GG_m$ action on this given by scaling the variable $t$. Thus it descends to a sheaf on $X \times \filstack$.

\end{defn}

\begin{rem}
By the definition of the Rees construction, we see that fiber of this object over $B \GG_m$ is the associated graded of the weight filtration, nameely $\Sym(\mathbb{T}_X)$. Meanwhile, pullback along the generic point recovers the ring of differential operators $\mathcal{D}_X$ itself.  Thus we may conclude that the stucture of a $\mathcal{D}^{\lambda}_{X}$-module is none other than that of a  $\lambda$-connection.  
\end{rem}

The following definition of the Hodge degeneration goes back to Simpson, following a suggestion of Deligne. It gives rise to a natural geometric parametrization for the notion of $\lambda$-connection. 

\begin{const} \label{Hodgestack}
Form the (levelwise) mapping stack over $\filstack$:
$$
\Map_{\filstack}(S^{0, \bullet}, X|_{\filstack}),
$$
where $S^{0,\bullet}$ is the cogroupoid object in derived stacks studied in Section \ref{cogroupoidcircle}. 
This will be a groupoid object in derived stacks. Next form the formal completion in each degree, along the map 
$$
X|_{\filstack} \to \Map_{\filstack}(S^{0, \bullet}, X|_{\filstack}).
$$
In more detail, we define this (in accordance with Section \ref{section formal completions}) by the following pullback square:
\begin{equation}
    \xymatrix{
& \Map_{\filstack}(S^{0, \bullet}, X|_{\filstack})^{\wedge}_{X|_{\filstack}}   \ar[d]^{} \ar[r]^-{} &  \Map_{\filstack}(S^{0, \bullet}, X|_{\filstack}) \ar[d]^{}\\
 & [X|_{\filstack}]_{dR} \ar[r]^-{} & [\Map_{\filstack}(S^{0, \bullet}, X|_{\filstack})]_{dR}
}    
\end{equation}
Notice that over the point $\Spec R \xrightarrow{1} \filstack$ the map with respect to which we are taking the formal completion is precisely the diagonal map $\Delta: X \to X \times X$; over the point $0 : B \GG_m \to \filstack$, this is exactly the inclusion $X \to \mathcal{T}^{\bullet}$ of the units into the group(oid) object defining the Dolbeault space. 
\end{const}

\begin{defn} \label{therealhodgestack}
We denote the above simplicial stack by  $X^{\bullet}_{\lambda}$. We then define 
$$
X_{Hod} = ||X^{\bullet}_{\lambda}||
$$
as its classifying stack. 
\end{defn}

By construction, the stack $X_{Hod}$ admits a map to $\filstack$. Pulling back along the generic fiber, one recovers the de Rham space:

\begin{prop}
There is an equivalence 
$$
X_{Hod} \times_{\filstack} \Spec R \simeq X_{dR}
$$
\end{prop}

On the other hand, pulling back along the map $\iota: B \GG_m \to \filstack$ recovers the Dolbeault space:

\begin{prop}
There is an equivalence 
$$
X_{Hod} \times_{\filstack} B \GG_m  \simeq X_{Dol},
$$
where the structure map $X_{Dol} \to B \GG_m$ arises from the canonical dilation $\GG_m$-action on the formal group $\widehat{\mathbb{T}X}$.
\end{prop}

\noindent The following theorem morally goes back to Simpson and may be proven similarly to Theorem \ref{theorem classic D module}:

\begin{thm}[Simpson]  
Let $X$ be a smooth scheme. Then there is an equivalence of $\infty$-categories 
$$
\QCoh(X_{Hod}) \simeq \Mod(\mathcal{D}_{X}^{\lambda}) 
$$
Here, the right hand side denotes the $\infty$-category of $\lambda$-connections.  
\end{thm}

\begin{rem}
One may  thus think of the $\infty$-category $\QCoh(X_{Hod})$ as giving a $\QCoh(\filstack)$-linear  enhancement of the $\infty$-category of $\mathcal{D}_X$-modules. By the Rees correspondence, this is none other than a filtration at the level of categories, which manifests itself as the well known filtration on the ring of  differential operators $\mathcal{D}_X$. The interested reader may consult  \cite[Remark 3.2.3]{toen2020foliationsgrr} for a generalization of this story to the setting of \emph{derived foliations}.
\end{rem}
\begin{const}
Let $X$ be a derived scheme. Recall from \cite{moulinos2019universal} that we may form the following derived mapping space 
$$
\mathcal{L}_{fil} X := \Map_{\on{dStk}_{\filstack}}(S^1_{fil}, X|_{\filstack})
$$
This is referred to in loc. cit. as the filtered loop space. By the work in Section \ref{cogroupoidcircle}, $S^1_{fil}$ fits as the degree $(1,1)$ piece of the 2-cogroupoid object $S^{1, \bullet}$. We take this into account by defining 
$$
\Map_{\on{dStk}_{\filstack}}(S_{fil}^{1, \bullet, \bullet}, X_{\filstack})
$$
which in turn, is a 2-groupoid in $\on{dStk}_{/ \filstack}$, over $X|_{\filstack}$. 
\end{const}

One has the following proposition:

\begin{prop} \label{loopsalreadycomplete}
Let $X$ be a derived scheme. Then for each $(m,n)$ the derived scheme $\mathcal{L}^{(m,n)}(X)$ is formally complete along $X|_{\filstack} \to \mathcal{L}^{(m,n)}(X)$. 
\end{prop}

\begin{proof}
In degree $(1,1)$, the cogroupoid object ${S^1}^{\bullet, \bullet}_{fil}$ is precisely $B \mathbb{H}$ where $\mathbb{H}$ is the filtered group scheme from \cite{moulinos2019universal}. After forming mapping spaces out of this, one obtains a derived scheme whose truncation is the truncation of $X$ itself. This uses the fact that $X$ is a scheme and does not exhibit any stacky behavior, and so there are no nonconstant maps $B \mathbb{H} \to X$. Thus the map $S^1_{Fil} \to *$ induces an equivalence on truncations $\mathcal{L}_{fil}X$, which is in particular a nil-isomorphism. 
Now that we know that the degree $1$ piece is formally complete, we use the intrinsic symmetries along the diagonal that follow from the $E_2$-groupoid structure, which allow for us to conclude that it is formally complete in each bidegree. Hence we conclude that this is a formal groupoid, in the sense of \cite{gaitsgory2020study}.
\end{proof}

Hence, $\mathcal{L}_{fil}(X)$ is a  $2$-groupoid  object in formal stacks. Moreover as a bisimplicial object, this has the constant simplicial diagram on $X|_{\filstack}$ on the zeroth row and column. Our next goal is to  compute the $2$-fold delooping of this groupoid, in the setting of formal moduli problems. Before passing to the formal context however, one can make the following observation:
 
 \begin{prop} \label{identification2groupoids}
There is an equivalence  
$$
\mathcal{L}^{\bullet, \bullet}_{fil}(X) \simeq Nerve(X|_{\filstack} \to \Map(S^{0, \bullet}_{fil}, X|_{\filstack}))
$$
of 2-groupoid objects in $\on{dStk}_{\filstack}$. 
\end{prop}

\begin{proof}
This will follow by explicitly identifying the two objects in $\on{Grpd}^{(2)}$. The argument ultimately boils down to the fact that both simplicial objects depend on the $E_2$-cogroupoid structure of $S^1_{fil}$ (more precisely, its structure as a cogroupoid object in the category of cogroupoid objects). At the level of objects, we fix $(n,m) \in \Delta^{op} \times \Delta^{op}$. Then in bisimplicial degree $(n,m)$, the object $\mathcal{L}^{\bullet, \bullet}_{fil}(X)$ is given by 
$$
\Map_{\filstack}((S_{fil}^1)^{n,m}, X|_{\filstack}).
$$
where
\begin{equation}
(S_{fil}^1)^{n,m}=  \overbrace{* \bigsqcup_{\underbrace{S^0_{fil} \sqcup_{*} \cdots \sqcup_{*} S^0_{fil}}_{n \, \on{  times} }}* \cdots * \bigsqcup_{\underbrace{S^0_{fil} \sqcup_{*} \cdots \sqcup_{*} S^0_{fil}}_{m \, \on{  times} }} *}^{m \, \on{  times} }.
\end{equation}
Note that $* =\filstack$ in this case as all these constructions are being performed in $\on{Stk}_{\filstack}$.  
Applying the functor $\Map_{\on{dStk}_{/ \filstack}}(-, X|_{\filstack})$ out of the above object sends this push-out square to a fiber product, by the argument in the proof of Proposition \ref{groupoidalnature}. A moment's thought shows that this is precisely the $(n,m)$th degree of the bisimplicial object 
$$
\on{Nerve}(X|_{\filstack} \to \Map(S^0_{fil}, X|_{\filstack}))
$$
\end{proof}

\noindent Passing to the setting of formal moduli problems relative to $X|_{\filstack}$ now gives rise to the following corollary:

\begin{cor}
There exists an equivalence
$$
B^{(1)} \mathcal{L}_{fil}^{\bullet, \bullet}(X) \simeq X^\bullet_\lambda 
$$
of formal groupoid objects. Here the left hand side denotes the delooping of the groupoid object $\mathcal{L}_{fil}(X)$ along the vertical or horizontal direction, and the right side is the formal groupoid whose classifying stack  is the Hodge degeneration.
\end{cor}

\begin{proof}

We first remark that the delooping of the $E_2$-groupoid will be a groupoid object in formal stacks, by Theorem \ref{gaitsgorymain}. Now, as we saw in Proposition \ref{identification2groupoids}, there is an identification of  $2$-groupoids 
$$
\mathcal{L}^{\bullet, \bullet}_{fil}(X) \simeq Nerve(X|_{\filstack} \to \Map(S^{0, \bullet}_{fil}, X|_{\filstack})).
$$
By Proposition \ref{loopsalreadycomplete}, the left hand side is already formally complete in each simplicial degree; thus we may conclude that the right hand side is a formal $2$-groupoid as well.

Next we apply the correspondence between formal groupoids and formal moduli problems under $X$, cf. Theorem \ref{gaitsgorymain}. This tells us that the $2$-groupoid $\mathcal{L}_{fil}^{\bullet, \bullet}(X)$ arises as the Čech nerve of some map $X|_{\filstack} \to \EuScript{Y}^{\bullet}$, which is a nil-isomorphism in the language of \cite{gaitsgory2020study}. We may identify this morphism of simplicial objects with the morphism which in simplicial degree $n$ is given by the map 
$$
X|_{\filstack} \to \Map( S^{0, n}_{fil}, X|_{\filstack})\widehat{_{X|_{\filstack}}}= X^\bullet_{\lambda},
$$
i.e. the map of $X|_{\filstack}$ to its formal thickening in $\Map( S^{0, n}_{fil}, X|_{\filstack})\widehat{_{X|_{\filstack}}}$. Thus, we have identified the $1$-groupoid object  $B \mathcal{L}_{fil}^{\bullet, \bullet}X$ with  $X_{\lambda}^{\bullet}$. Note that the structure maps agree by construction; indeed $X_{\lambda}$ was defined in Construction \ref{Hodgestack} precisely as the levelwise formal completion of the groupoid object $\Map(S^{0,\bullet}_{fil}, X|_{\filstack})$ along the map 
$$
X|_{\filstack} \to \Map(S^{0,\bullet}_{fil}, X|_{\filstack}). 
$$
\end{proof}

We may now put this all together and quickly prove the main theorem, which we restate for the reader's convenience. 
\begin{thm}
There exists an equivalence 
$$
B^2 \mathcal{L}^{fil}(X) \simeq X_{Hod} 
$$
\end{thm}

\begin{proof}
We can compute $B^2 \mathcal{L}^{fil}(X)$ by first delooping in the vertical direction and then in the horizontal direction (or in reverse). As we saw in the previous proposition, the stage one delooping recovers the formal groupoid $X^{\bullet}_{\lambda}$. By Definition \ref{therealhodgestack}, the delooping of this is precisely the Hodge stack $X_{Hod}$.  
\end{proof}

\section{Group structures on the loop space and the Todd class}
Let $f: X \to Y$ be a proper morphism of schemes. A modern reinterpration of the  Grothendieck Riemann-Roch theorem encodes the compatibility between the Chern character and the proper pushforward on Hodge cohomology, which arises from the pushforward $f_{*}: \Perf(X) \to \Perf(Y)$. As described in \cite[Section 5.2]{hoyois2021categorified}, the HKR equivalence 
$$
R\Gamma(\OO_{\mathcal{L}X}) \simeq \Sym(\Omega_{X/k}[1])
$$
intertwines the ``integration map"  $\OO(\mathcal{L}(X)) \to \OO(\mathcal{L}(Y)$ with the pushforward map on Hodge cohomology, twisted by the  square root of the \emph{Todd class}:

\begin{equation*}
 \xymatrixcolsep{5pc} \xymatrix{
& K_0(X) \ar[d]_{\on{ch}} \ar[r]^{f_*} &  K_0(Y) \ar[d]_{\on{ch}}\\
 & \oplus_i H^i(X, \Omega_X^{i})  \ar[r]^{f_*( \cup \sqrt{\on{td}(\mathbb{T}_f)}} & \oplus_i H^i(Y, \Omega_Y^{i})
}   
\end{equation*}

In this section we give a conceptual explanation for the Todd class in terms of the $2$-cogroupoid structure on $S^1_{fil}$ studied in this paper. In particular we will see that it arises from the failure of $S^1_{fil}$ to be a constant degeneration of cogroup(oid) objects. 

\begin{const}\label{groupstructures}
Recall from above that $2$-cogroupoid structure on $S^1_{fil}$ gives the filtered loop space  $\mathcal{L}_{fil}(X)$a $2$-groupoid structure. In particular, it may be viewed as an $E_2$-group object over the derived scheme $X$ as $X$ will be the degree $(0,0)$ space of objects. 
\end{const}
 
We would like to study the corresponding group stuctures on $\mathcal{L}_{fil}(X)$ upon specializing along the closed and generic points of $\filstack$. Recall from \cite{moulinos2019universal}, that over $\eta : \Spec k \to \filstack$, there is an equivalence 
$$
(S^1_{fil})^{u} := \eta^*(S^1_{fil}) \simeq B \GG_a
$$
of group stacks. Similarly one has an equivalence 
$$
S^1_{gr} := \iota^*(S^1_{fil}) \simeq B \GG_a
$$
We remind the reader that this is only true in characteristic zero. Thus one has an HKR equivalence (cf. \cite{toen2011algebres, ben2012loop})
$$
\on{exp}: \shiftedtangent \simeq  \loopsX  
$$

We will see that there exist two group structures on $\mathbb{T}_X [-1]$, one which is related to the cogroupoid structure on $S^1_{gr}$, and the other which is related to that on $S^1$. The extent to which these are nonequivalent is detected by a distinguished class in 
$$
\pi_0 \Gamma(\shiftedtangent, \OO_{\shiftedtangent})^\times \cong \oplus_i H^i(X, \Omega^i_X)
$$ 
This is none other than the Todd class. 

\begin{rem}
These group objects  $\mathcal{L}(X)$ and $\mathbb{T}_{X}[-1]$ over $X$ are in fact formally complete along $X$. Thus we may view them as group objects in formal moduli problems over $X$.  
\end{rem}
\subsection{Formal groups and the Todd class} \label{section formal groups todd}
We now review how a group structure gives rise to an orientation of the canonical bundle on $\mathbb{T}_X [-1]$. Much of the following discussion is taken from \cite{kondyrev2019equivariant}. 

\begin{const} \label{trivializationtangent}
We first review, following \emph{loc. cit.} the canonical  trivialization of the relative tangent bundle of a formal group over $X$. Let $\formalgroup$ be the formal group in question. In this setting one has the following trivialization 
$$
\mathbb{T}_{\formalgroup/ X} \simeq \pi^* \mathbb{T}_{X / B_{X} \formalgroup} \simeq \pi ^* e^* \pi^* \mathbb{T}_{X / B_{X} \formalgroup} \simeq \pi ^* e^* \mathbb{T}_{\formalgroup /X} \simeq \pi ^* \mathfrak{g},
 $$
 where $\mathfrak{g} := \on{Lie}_{X}(\formalgroup)$ is the corresponding Lie algebra of $\formalgroup$. One uses here the key property of the relative tangent sheaf being stable under pullbacks, cf. \cite[Proposition 5.1.8]{kondyrev2019equivariant}. 
\end{const}

\begin{rem}
Let $\formalgroup = \mathcal{L}X$. Using the above, one obtains an orientation 
$$
\omega_{\formalgroup} \simeq \omega_{\formalgroup / X} \otimes \pi^*(\omega_X) \simeq \pi^* (\omega_{X}^{-1}) \otimes \pi^*(\omega_X)   \simeq \OO_{\formalgroup}
$$

\end{rem}

\begin{const}
Recall the construction of the determinant of a perfect complex from \cite{toendeterminants}. This is defined as a morphism of stacks 
$$
\on{det}: \mathbf{Perf} \to \mathbf{Pic}
$$
where the left hand side  is the derived stack classifying perfect complexes and the right hand side classifies invertible objects. Now, we fix a formal derived stack $\EuScript{Y}$ over $X$ whose relative tangent complex is perfect, and on which one may equip  two distinct formal group structures $g_1, g_2$. By composing the trivialization of $\mathbb{T}_{\EuScript{Y}}$ arising from $g_1$ with the inverse of that arising from $g_2$, we obtain an automorphism
$$
\gamma: \mathbb{T}_{\EuScript{Y}/ X} \xrightarrow{\iota_{u}} \pi^*(\mathfrak{g}) \xrightarrow{\iota_{gr}} \mathbb{T}_{\EuScript{Y}/ X}
$$
This is an endomorphism of the relative tangent complex of $\Y$ over $X$. 
We now define 
$$
\on{td}^{grp}(\EuScript{Y}):= \on{det}(\gamma) \in \pi_0\Gamma(\EuScript{Y}, \OO_{\EuScript{Y}})^{\times},
$$
by way of the induced map $\on{det}_{\EuScript{Y}}: \mathbf{Perf}(\EuScript{Y}) \to \mathbf{ Pic}(\EuScript{Y})$  
\end{const}

\begin{rem}
Let us remark for the sake of clarity that $\on{det}(\gamma)$ is in fact an invertible element of $\Map_{\OO_{\EuScript{Y}}}(\pi^* \omega_X^{-1},\pi^* \omega_X^{-1} )$. However, note that this canonically equivalent to $\Map_{\OO_{\EuScript{Y}}}(\OO_{\EuScript{Y}},\OO_{\EuScript{Y}}) \simeq \OO_{\EuScript{Y}}$. Thus we obtain a well defined invertible  element of $\pi_0\Gamma(\EuScript{Y}, \OO_{\EuScript{Y}})^{\times}$

\end{rem}

\begin{rem}
Let $X$ be a derived scheme and  fix $\EuScript{Y}= \loopsX \simeq \shiftedtangent$. This obtains two group structures, one arising from the cogroup structure on $(S^1_{fil})^{u}$, and the other coming from the cogroup structure on the $S^1_{gr} = \Spec^{\Delta}(k \oplus k[-1])$. Then 
$$
\on{td}^{grp}(\shiftedtangent) \in \pi_0 \Gamma(\shiftedtangent, \OO_{\shiftedtangent})^{\times} \cong (\bigoplus_{i} H^i(X, \Omega^i_{X}))^{\times}
$$
\end{rem}

In \cite{kondyrev2019equivariant} it is shown that this recovers the Todd class of a scheme. We briefly recall a broad explanation for why this is true.  

\begin{prop}[\cite{kondyrev2019equivariant}]
Let $X$ be a smooth and proper scheme. Then the group theoretic Todd class, defined above, recovers the classical Todd class of the Lie algebroid $\mathbb{T}_X$. 
\end{prop}

\begin{proof}
The loop group structure on $\mathcal{L}(X)$ makes it into a formal group over $X$. The result follows from a general statement valid for arbitrary formal groups. For a general  formal group $\formalgroup \to X$, let $\mathfrak{g} := \on{Lie}(\formalgroup)$. There exists a general equivalence of formal moduli problems over $X$:
$$
\exp: \mathbb{V}(\mathfrak{g})  \to  \formalgroup
$$
where the right hand side denotes the formal vector bundle stack associated to $\mathfrak{g}$. Via this equivalence, $\mathbb{V}(\mathfrak{g})$ inherits 2 group structures (one being the abelian one, the other via transport of structure from $\formalgroup$). Via the discussion above, this gives rise to two trivializations of the relative tangent $\mathbb{T}_{\formalgroup / X}$ which is denoted suggestively in \cite{kondyrev2019equivariant} as 
\begin{equation} \label{saywatup}
    d\on{exp}_{\formalgroup}: \pi^* \mathfrak{g} \to \pi^* \mathfrak{g} 
\end{equation}

\noindent The determinant (in the sense of \cite{toendeterminants}) of this automorphism then gives the group theoretic Todd class. 

Meanwhile, the Todd class as it appears in the statement of the GRR theorem, is given as a multiplicative characteristic class; it is given by the formula 

$$
\on{td}_X= \det(f(At(X)))
$$
where $f(x)$ is the formal power series: 
$$
f(x) = \frac{1-e^{-x}} {x}
$$
The key result of \cite{kondyrev2019equivariant} states that for an arbitrary formal group $\formalgroup$, the automorphism (\ref{saywatup}) may be expressed as 
$$
   d\on{exp}_{\formalgroup} = \frac{1-e^{- \on{ad}_\mathfrak{g}}} {\on{ad}_\mathfrak{g}},
$$
where $\on{ad}_\mathfrak{g}$ denotes the adjoint representation of the Lie algebra $\mathfrak{g}$, given by the Atiyah class of $\mathfrak{g}$. 
\end{proof}

\subsection{Non-formality of the pinch map}
We would like to relate the construction of the Todd class above to the cogroupoid structures on $S^1_{fil}$. We will see that the data of the Todd class is in a precise sense included in the data of the $E_2$-cogroupoid structure on $S^1_{fil}$. In particular, the nontriviality of the Todd class will follow from the $E_1$ cogroup structure on the topological circle $S^1$, in the $\infty$-category of pointed spaces $\mathcal{S}_*$ This cogroup structure is often described by the well-known ``pinch map"
$$
S^1 \to S^1 \vee S^1
$$
of pointed spaces. 
The goal of this section is to relate the existence of this Todd class with the failure of the resulting cogroup(oid) structure on $S^1$ to be \emph{formal}. We first explain what we mean by this.

\begin{defn}
Let $\EuScript{X} \to \filstack$. Let $\EuScript{X}_1 := \eta^*(\EuScript{X})$. We say  that $\EuScript{X}$ is a \emph{constant degeneration} (of $\EuScript{X}_1$) if there exists an equivalence 
$$
\EuScript{X}_1 \simeq \EuScript{X}_0,
$$
where the right hand side denotes the pullback of $\EuScript{X}$ along the composite map
$$
\Spec k \xrightarrow{\pi} B \GG_m \xrightarrow{\iota} \filstack
$$ 
Similarly let $\EuScript{X}^{\bullet , \cdots, \bullet}$ be $E_n$ cogroupoid object in affine stacks. Then it is a \emph{constant degeneration of cogroupoid objects} if there is an equivalence 
$$
\EuScript{X}_1^\bullet \simeq \EuScript{X}_0^\bullet.
$$
of cogroupoids.
\end{defn}

\begin{rem}
 By   \cite[Proposition 4.5.8]{raksit2020hochschild}, the Postnikov filtration functor  induces a fully faithful embedding $\tau_{\geq * }:  \mathsf{coSCR}_R \hookrightarrow \on{Fil}(\mathsf{coSCR}_R )$. In particular, if $X$ is a topological space, its cochain algebra $R^{X} =C^*(X, R)$ can be promoted, by way of the Postnikov filtration, to a filtered commutative algebra. This moreover degenerates to the cohomology ring $H^*(X, R)$ (viewed as a dga with zero differential) at the level of associated graded.     
\end{rem}

\begin{defn}
Let $A$ be a cosimplicial commutative algebra.  By the above remark, it admits a canonical lift to filtered cosimplicial commutative algebras. By the Rees construction, we may in turn view this as a  cosimplicial commutative algebra in $\QCoh(\filstack)$.  We say that $A$ is formal if the affine stack $\Spec^{\Delta}(A)$ is a constant degeneration  over $\filstack$. Similarly, if  $A_{\bullet}$ is a groupoid object in $\mathsf{coSCR}_R$, then it is formal as a groupoid object if $\Spec^{\Delta}(A)^{\bullet} \to \filstack$ is a constant degeneration of cogroupoids. 
\end{defn}

\begin{rem}
Let $X^{\bullet, \cdots, \bullet}$ be an $E_n$-cogroup(oid) object in spaces. Then, again using \cite[Proposition 4.5.8]{raksit2020hochschild}, we obtain an $E_n$-cogroup(oid) object in affine stacks $\EuScript{X}^{\bullet, \cdots, \bullet}$ over $\filstack$, such that in degree $(1, \cdots, 1)$, there is an equivalence 
$$ 
\EuScript{X}_{1} \simeq \Spec^{\Delta}(C^*(X, k)), \, \, \, \, \EuScript{X}_{0} \simeq \Spec^{\Delta}(H^*(X, k))
$$
\end{rem}

We remark that the $E_2$-cogroupoid structure on $S^1$ discussed thus far contains strictly more structure than the cogroup object $S^1$ in pointed spaces. 

\begin{prop} \label{determination of structures}
The co-group structure on $S^1$ viewed as a pointed space is determined by  the $2$-cogroupoid structure on $S^1$ in the $\infty$-category of (unpointed) spaces. In particular the cosimplicial object describing the cogroup $S^1$ can be recovered as the projection to the $1$st column of $S^{1, \bullet, \bullet}$.  
\end{prop}

\begin{proof}
Recall that 2-cogroupoid object $S^1$ is defined by as the iterated conerve of the map $\emptyset \to \pt $. In particular, this will be the conerve of the map of  cosimplicial spaces  
$$
S^{0, \bullet} \to \pt,
$$
where $\pt$ is viewed as a constant cosimplicial space, and the left hand side is given by $S^0 \sqcup \cdots \sqcup S^0 $ in each degree.  This begets a bicosimplical object, which we see by Section \ref{section cogroupoid objects} is an $E_2$-cogroupoid.  Restricting to the first column (or the first row as this bicosimplicial object will be symmetric along the diagonal) piece, one obtains the cosimplicial object given by the conerve of the map
$$
S^0 \to \pt,
$$
which in degree $n$ is precisely given by $S^1 \vee \cdots \vee S^1$. This exactly  encodes the cogroup structure of $S^1$ as an object in the $\infty$-category of pointed spaces.  
\end{proof}

\begin{cor}
The 2-cogroupoid structure on $S^1$ determines the (formal) group structure on $\mathcal{L}X$ over $X$. 
\end{cor}

\noindent Now, let us fix again a derived scheme $X$, and turn our attention to the filtered loop space $\mathcal{L}_{fil}X$. Note that this is a constant degeneration over $\filstack$,  of the loop space $\mathcal{L}_{fil}$ to the vector bundle stack corresponding to the (shifted) Lie algebra $\mathbb{T}_{X}[-1]$. The 2-cogroupoid structure on $S^1_{fil}$, gives this the structure of an $E_2$-groupoid object in formal moduli problems over $X|_{\filstack}$. Over the generic point $\eta: \Spec k \to \filstack$, this recovers the group structure on $\mathcal{L} X$ and over $B \GG_m$, this recovers the abelian group structure on $ \mathbb{T}X[-1]$ arising from its structure as the (formal completion) of a linear stack.

\begin{rem}
Let $X = S^1$ equipped with its $E_1$-cogroupoid structure. The formality as above is equivalent to the question of whether or not $S^1_{fil}$ is a constant degeneration of $E_1$-cogroupoid objects over $\filstack$.  
\end{rem}

We already know that as an $E_2$-cogroupoid, the degeneration is nonconstant: 

\begin{prop}
The $E_2$-cogroupoid object ${S^{1,\bullet, \bullet}_{fil}}$ is not a constant degeneration over $\filstack$.

\end{prop}

\begin{proof}
We apply mapping spaces into a derived scheme $X$. This gives rise to an $E_2$-(formal)-groupoid over $X|_{\filstack}$. By the main theorem, $B^{(2)} \mathcal{L}_{fil}(X) \simeq X_{Hod}$. This obviously is not a constant degeneration of formal moduli problems  since 
$$
X_{dR} \not\simeq X_{Dol}
$$
Hence we may conclude that ${S^{1,\bullet, \bullet}_{fil}}$ is not formal as an $E_2$-cogroupoid object.
\end{proof}

In fact, the failure of $\mathcal{L}_{fil}(X)$ to be a constant degeneration of groups,i.e. the failure of the formality of the pinch map $S^1 \to S^1 \vee S^1$ gives rise functorially to the  existence of  non-trivial Todd classes for smooth and proper schemes. 

\begin{thm}
The existence of non-trivial Todd classes implies that the cogroupoid  (equivalently pointed cogroup) structure on $S^1$ corresponding to the pinch map is not formal.  
\end{thm}

\begin{proof}
The $E_1$-cogroupoid ${S^{1, \bullet}_{fil}}$ provides a degeneration of pointed $E_1$-cogroup objects over $\filstack$; this specializes to the cogroup $(S^1_{fil})^u$ over the generic fiber and to $S^1_{gr}$ over the special fiber. Note that one recovers the filtered circle in the degree 1 stage. Recall that in characteristic zero at the level of stacks, this is a constant degeneration, so we must verify that the cogroupoid structures are themselves different. 

Before doing so, we remark  that the cogroup structure on $(S^1_{fil})^u  = S^1_{fil} \times_{\filstack} \Spec k$ is indeed controlled by the cogroupoid (equivalently pointed cogroup) structure on the topological space $S^1$. The constant stack functor
$$
\mathcal{S}_* \to  \on{Stk}_k  
$$
being a left adjoint, preserves colimits, and thus the associated cosimplicial object in  $\on{Stk}_k$ is also an $E_1$-cogroup. By applying affinization, we obtain a cogroup object in \emph{affine stacks}; however, the affinization morphism $S^1 \to \on{Aff}(S^1)$ is easily seen to be a morphism of cogroups so the ``unipotent" loop space $\Map( (S^1_{fil})^{u}, X)$ (which recovers $\mathcal{L}(X)$) inherits the same $E_1$ group structure over $X$ as $\mathcal{L}(X)$

Let us fix $X$ a derived scheme. The HKR theorem in characteristic zero gives rise to an equivalence 

$$
 \mathcal{L}(X) \simeq \Map((S^1_{fil})^u, X) \simeq  \Map((S^1_{gr})^u, X) \simeq \mathcal{T}X[-1]
$$
of derived schemes over $X$. As we remarked earlier, these are in fact formally complete along $X$, so that they may be viewed as formal moduli problems over $X$. These sit in the degree $1$ piece of the simplicial object in formal moduli problems over $X$ classifying the loop group and abelian group structures on $\mathcal{T}X[-1]$ respectively. These are given by  $\mathcal{L}_{gr}(X)^{\bullet, \bullet}$ and $\mathcal{L}(X)^{\bullet, \bullet}$. 

The claim is that the Todd class measures the difference between these two group structures, in a functorial way. In particular  one may express the assignment 
$$
X \to \on{td}(X)
$$
as the image of $\mathcal{L}^{\bullet, \bullet}_{fil}(X)$, with its group structure induced by mapping objects from the cogroup structure on $S^{1,\bullet}_{fil}$, under the following composition (of core $\infty$-groupoids): 
\begin{align*}
  & \on{Grp}( \filteredloopspaceX)\rightarrow \on{Grp}(\mathcal{T}_X[-1]) \times \on{Grp}(\mathcal{T}_X[-1])  \\     \to& \on{Iso}_{\QCoh(\mathcal{T}_X[-1])}(\mathbb{T}_{\mathcal{T}_X[-1]/X}, \pi^*(\mathbb{T}_X[-1])) \times  \on{Iso}_{\QCoh(\mathcal{T}_X[-1])}(\mathbb{T}_{\mathcal{T}_X[-1]/X}, \pi^*(\mathbb{T}_X[-1])) \\
 \to & \on{Aut}_{\QCoh(\mathcal{T}_X[-1])}(\mathbb{T}_{\mathcal{T}_X[-1]/X}) \xrightarrow{\mathbf{det}} \OO_{\mathcal{T}_X[-1]}^{\times}
\end{align*}
Here the first map is induced by simultaneously pulling back the group structure to the generic and special fiber to get two different group stuctures on $\mathcal{T}_X[-1]$ (since we are working in characteristic zero, both the generic and special fibers are equivalent to $\mathcal{T}_X[-1]$). The second arrow is the assignment, to each of these group structures, of a trivialization of the relative bundle. As described in \cite[Construction 3.3.1]{kondyrev2019equivariant}, this assignment is canonical, hence the functoriality of the second arrow. The third arrow is just the composition of the two equivalences which gives an automorphism of the relative tangent bundle. 

Now, let $X$ be in particular smooth and proper. As we reviewed above in Section \ref{section formal groups todd}, the group theoretical Todd class agrees with the classical Todd class. Thus, the 1-cogroupoid $S^{1, \bullet}_{fil}$  gives rise to a degeneration of groupoid objects, natural in $X$, with special fiber the abelian group structure and generic fiber given by the loop group. The fact that there exists some smooth and proper scheme $X$ for which $\on{td}(X) \neq 1$, implies that this is not a constant degeneration.
Thus, the cogroup structures on $S^1$ and $S^1_{gr}$ are themselves not equivalent.
We may conclude from all this that the $E_1$-cogroup in pointed spaces $S^1$ is not formal.
\end{proof}

\section{Consequences for Hochschild cohomology} \label{section Hochschild cohomology}
We would like to end with a few further remarks on the $E_2$-cogroupoid structure on $S^1_{fil}$ together with consequences at the level of Hochschild \emph{cohomology}.

This is presumably well known to experts, but we believe it is worthwhile explicitly relating the $E_2$-algebra structure on Hochschild cohomology along with its compatible HKR filtration to the cogroupoid objects in stacks presented here. 

\begin{defn}
Let $X$ be a derived scheme. The Hochschild cohomology sheaf is defined to be 
$$
\mathcal{HH}^*(X) = {p_1}_* \underline{End}_{\OO_X \otimes \OO_X}(\OO_X) 
$$
This acquires an $\OO_{X}$-linear  $E_1$-algebra structure. Furthermore, by e.g. \cite[Section 1.4]{markarian2009atiyah}, there exists a perfect pairing 
$$
\mathcal{HH}^*(X) \otimes_{\OO_X} \OO_{\loopspaceX} \to \OO_X,
$$
displaying $\mathcal{HH}^*(X)$ as the $\OO_X$-linear dual of $\OO_{\loopspaceX}$. 
\end{defn}

\begin{prop}
The $E_2$-cogroupoid structure on $S^1$ (resp. $S^1_{fil}$) endows $\on{HH}^*(X)$ with the structure of an $E_2$-algebra in $\Mod_k$, resp $\on{Fil}(\Mod_k)$.  
\end{prop}

\begin{proof}
 We will prove this in the unfiltered case, but all the claims go through mutatis mutandis in the filtered setting. Let us focus our attention of the structure sheaf of $\mathcal{L}X$; the induced $E_2$-groupoid structure over $X$ will endow $\OO_{\loopspaceX}$ with an $E_2$-coalgebroid structure over $\OO_X$. If we take $\OO_X$-linear duals, this gives an $E_2$-algebroid strucure on  $\OO_{\mathcal{L}X}^\vee$. This bisimplicial object encodes a $\OO_X$-linear multiplication, and a $(\OO_X \otimes_{k} \OO_X)$-linear homotopy  corresponding to the $E_2$-commutativity. We may forget this to a $k$-linear homotopy; this gives $\on{HH}^*(X)$ the structure of an $E_2$-algebra in  $\Mod_k$.  
\end{proof}

\noindent As it turns out, the standard $E_2$-algebra structure on  (filtered) Hochschild cohomology may be recovered from the $E_1$ cogroupoid  $S^{0, \bullet}_{fil}$:

\begin{rem}
There exists an equivalence
$$
\QCoh(X \times X) \simeq \on{Fun}_{k}(\QCoh(X), \QCoh(X)) 
$$
The left hand side inherits a monoidal structure from the right, by composition. This is often referred to as the convolution  monoidal structure.
\end{rem}

\begin{prop}
The  monoidal structure on $\QCoh(X \times X)$ is induced by the $E_1$-cogroupoid structure on $S^{0, \bullet}$.
\end{prop}

\begin{proof}
Following \cite[Section 5.2]{gaitsgory2019study} the assignment  $\EuScript{X} \mapsto \QCoh(\EuScript{X})$ can be expressed as a symmetric monoidal functor from a certain $(\infty, 2)$-category of correspondences (in stacks) to the $(\infty,2)$-category of $k$-linear stable $\infty$-categories. 
As a consequence, for any groupoid object $\EuScript{M}_\bullet$, one obtains a monoidal structure on $\QCoh(\EuScript{M}_1)$.  
This discussion now applies to the $E_1$-groupoid $\Map(S^{0, \bullet}, X)$, which inherits the groupoid structure from the cogroupoid $S^0$. The degree 1 piece is
$$
\Map(S^0, X) \simeq X \times X
$$
so we conclude that $\QCoh(X \times X)$ admits an $E_1$-algebra structure. The equivalence
$$
\QCoh(X \times X) \simeq \on{Fun}_{k}(\QCoh(X), \QCoh(X)) 
$$
may now be upgraded  to an equivalence of monoidal $\infty$-categories by \cite[Remark 4.11]{ben2010integral}. 
\end{proof}

\begin{rem}
It is well known that one can now recover Hochschild cohomology of $X$,  as the endomorphisms of the unit with respect to this monoidal structure:
$$
\on{HH}^*(X) := \End_{\QCoh(X \times X)^{\otimes}}(\mathbf{1}),
$$
the unit here being the image of $\OO_X$ under pushforward along the diagonal $\Delta: X \times X$. In fact, as this is endomorphisms of the unit in an $E_1$-monoidal $\infty$-category, this is naturally an $E_2$-algebra, thereby recovering the well known fact that $\on{HH}^*(X)$ admits an $E_2$-algebra structure. 

\end{rem}
In fact the above argument applies verbatim to give a monoidal structure on 
$$
\QCoh(\Map_{\filstack}(S^0_{fil}, X|_{\filstack}))
$$
which specializes, upon pullback along $\Spec k \to \filstack$, to that of $\QCoh(X \times X)$. 

\begin{cor} \label{duven}
There exists a filtration on Hochschild cohomology compatible with its $E_2$-algebra structure, with associated graded the $E_\infty$-algebra $\Sym(\mathbb{T}_{X}[-1])$ 
\end{cor}

\begin{proof}
By the discussion in the proof of Proposition 8.4, the $\infty$-category $\QCoh(\Map(S_{Fil}^0,  X|_{\filstack}))$ acquires a monoidal structure. Moreover, this monoidal structure is linear over the base, in that it is $\QCoh(\filstack)$-linear. Thus one may define the enriched (over $\QCoh(\filstack) \simeq \on{Fil}(\Mod_R)$) endomorphism object 
$$
\on{HH}^*_{fil}(X) :=  \End_{\QCoh(\Map_{\filstack}(S^0_{fil}, X|_{\filstack}))}(\mathbf{1}) \in \QCoh(\filstack) \simeq \on{Fil}(\Mod_R)
$$
As endomorphisms of the unit in an $E_n$-monoidal category acquire an $E_{n+1}$ monoidal structure, this gives rise to the $E_2$-monoidal structure on $\on{HH}_{fil}^*(X)$. This base changes to $\on{HH}^*(X) \simeq \End_{\QCoh(X \times X)^{\otimes}}(\mathbf{1})$  upon passing to the generic fiber, thus recovering the $E_2$-monoidal structure on Hochschild cohomology.  

Let us now pass to the central fiber. By Koszul duality, together with the fact that 
$$
\Map(\Spec(k[ \epsilon]/ (\epsilon^2), X) \simeq \Spec \Sym(\mathbb{L}_X)  \simeq \mathcal{T}_X,
$$
one obtains an equivalence 
$$
\End(\mathbf{1}) \simeq \Sym^*(\mathbb{T}_{X}[-1]) 
$$
\end{proof}

\begin{rem}
We remark that we have now exhibited two $E_2$-algebra structures on Hochschild cohomology (and its filtered variant); one defined via the $E_2$ cogroupoid $S^{1, \bullet, \bullet}$ and the standard one defined by composition of natural transformations, now seen to be induced by the $E_1$-cogroupoid $S^{0, \bullet}$.  We claim that they are equivalent, but we do not include an argument here.  
\end{rem}

\begin{rem}
The same argument as that of Corollary \ref{duven} goes through to show that $E_n$-Hochschild cohomology,
$$
\on{HH}_{E_n}^*(X) :=  \End_{\QCoh(\mathcal{L}^{(n-1)}(X))}(\mathbf{1}) 
$$
obtains an $E_{n+1}$ algebra structure. For example, $\QCoh(\mathcal{L}(X))$ will obtain an $E_2$-monoidal structure, with unit given by the pushforward $e_*(\OO_X)$ along the constant morphism
$$
e: X \to \mathcal{L}(X)
$$
Hence we recover an $E_3$-algebra structure on $\on{HH}^*_{E_2}(X)$.

We remark that an alternative but closely related explanation for the $E_{n+1}$ structure on $n$-iterated Hochschild cohomology may be found in \cite{toenbranes}. 
\end{rem}


\bibliographystyle{amsalpha}
\bibliography{henven}
\end{document}